\documentclass[11pt]{amsart}
\usepackage{amsmath,amscd,amssymb,amsfonts,amsthm}
\usepackage[all]{xy}

 \newtheorem{theorem}{Theorem}[section]

\newtheorem{proposition}[theorem]{Proposition}

\newtheorem{lemma}[theorem]{Lemma}
\theoremstyle{remark}
\newtheorem{remark}[theorem]{Remark}
\newtheorem{definition}[theorem]{Definition}
\newtheorem{remarks}[theorem]{Remarks}
\newtheorem{example}[theorem]{Example}
\newtheorem{examples}[theorem]{Examples}
\renewcommand\L{\mathcal{L}}

\newcommand{\F}{\mathcal{F}}

\newcommand{\R}{\mathbb{R}}

\newcommand{\sC}{\mathsf{C}}

\newcommand{\sW}{\mathsf{W}}

\newcommand{\Z}{\mathbb{Z}}

\newcommand{\pr}{\on{pr}}

\newcommand\lie[1]{\mathfrak{#1}}
\renewcommand{\k}{\lie{k}}

\newcommand{\g}{\lie{g}}

\newcommand{\on}{\operatorname}

\newcommand{\Hom}{ \on{Hom}} 
\renewcommand{\ker}{ \on{ker}}

\newcommand{\Mult}{{\on{Mult}}}

\newcommand{\D}{ \mathcal{D} }

\renewcommand\a{\mathsf{a}}

\newcommand{\hra}{\hookrightarrow}

\newcommand{\xla}{\xleftarrow}

\newcommand{\rra}{\rightrightarrows}

\renewcommand{\d}{{\mathrm{d}}}

\newcommand{\ol}{\overline}

\newcommand\eps{\epsilon}
\newcommand\Om{\Omega}
\newcommand\om{\omega}

\newcommand{\f}{\frac}

\renewcommand{\H}{\ca{H}}

\renewcommand{\l}{\langle}
\renewcommand{\r}{\rangle}
\newcommand{\hh}{{ \f{1}{2}}}

\newcommand\pt{\on{pt}}

\newcommand\beqn{\begin{equation}}
\newcommand\eeqn{\end{equation}}
\newcommand{\ca}{\mathcal}
\newcommand{\wh}{\widehat}
\newcommand{\wt}{\widetilde}
\newcommand{\mf}{\mathfrak}
\newcommand{\beq}{\begin{eqnarray*}}
\newcommand{\eeq}{\end{eqnarray*}}

\newcommand{\sz}{\mathsf{s}}
\newcommand{\tz}{\mathsf{t}}

\newcommand{\botimes}{\bar{\otimes}}

\begin{document}

\author{D. Li-Bland}
\address{Department of Mathematics, University of California at Berkeley, 1071 Evens Hall, Berkeley, CA 94720, USA}
\email{libland@math.berkeley.edu}

\author{E. Meinrenken}
\address{University of Toronto, Department of Mathematics,
40 St George Street, Toronto, Ontario M4S2E4, Canada }
\email{mein@math.toronto.edu}

\title{On the Van Est homomorphism for Lie groupoids}
\begin{abstract}
The Van Est homomorphism for a Lie groupoid $G\rra M$, as introduced by Weinstein-Xu, is a cochain map  from the complex 
$C^\infty(BG)$ 
of  groupoid cochains to the Chevalley-Eilenberg complex $\sC(A)$ of the Lie algebroid $A$ of $G$. It was generalized by 
Weinstein, Mehta, and Abad-Crainic to a morphism from the Bott-Shulman-Stasheff complex $\Om(BG)$ to a 
(suitably defined) Weil algebra $\sW(A)$. In this paper, we will give an approach to the Van Est map in terms of the
Perturbation Lemma of homological algebra. This approach is used to establish the basic properties of the Van Est map. 
In particular, we show that on the normalized subcomplex, the Van Est map restricts to an algebra morphism. 
\end{abstract}
\date{\today}
\maketitle

\tableofcontents
\section{Introduction}
In their 1991 paper, Weinstein and Xu \cite{wei:ext} described an important generalization of the classical Van Est map \cite{vanEst:algCoh,vanEst:cartLeray,vanEst:groupCoh} to arbitrary
Lie groupoids $G\rra M$. Recall that the complex of groupoid cochains for $G$ consists of smooth functions on the space $B_pG$ of \emph{$p$-arrows}, that is, $p$-tuples of elements of $G$ such that any two successive elements
are composable. Its infinitesimal counterpart is the 
Chevalley-Eilenberg complex $\sC^\bullet(A)=\Gamma(\wedge^\bullet A^*)$ of the Lie algebroid of $G$.  
The generalized van Est map is a morphism of cochain complexes 
\begin{equation}\label{eq:weinxu}
\on{VE}\colon C^\infty(B_\bullet G)\to \sC^\bullet(A).
\end{equation}
Weinstein and Xu define this map in terms of the following formula, for $f\in C^\infty(B_p G)$ and $X_1,\ldots,X_p\in\Gamma(A)$, 
\begin{equation}\label{eq:expl1}
 \i(X_p)\cdots \i(X_1)\on{VE}(f)=\iota^*\sum_{s\in\mf{S}_p} \on{sign}(s) \L(X_{s(1)}^{1,\sharp})\cdots \L(X_{s(p)}^{p,\sharp}) f.
\end{equation}
Here the $X^{i,\sharp}$ for $X\in\Gamma(A)$ are the generating vector fields for certain commuting $G$-actions on $B_p G$, and $\iota\colon M\to B_pG$ is the inclusion as trivial $p$-arrows. 

Weinstein and Mehta \cite{meh:sup} indicated a generalization of \eqref{eq:weinxu} to a morphism of bidifferential complexes, 
\begin{equation}\label{eq:VEW}
\on{VE}\colon \Om^\bullet(B_\bullet G)\to \sW^{\bullet,\bullet}(A),
\end{equation}
from the Bott-Shulman-Stasheff double complex (i.e. the de Rham complex of  the simplicial manifold $B_\bullet G$) to a 
certain \emph{Weil algebra} of the Lie algebroid $A$. Their theory was formulated within the framework of supergeometry. 
Abad and Crainic \cite{aba:wei} gave a different construction of the Weil algebra and the Van Est map in terms of classical geometry, 
using \emph{representations up to homotopy}. Generalizing a result of Crainic \cite{cra:dif}, 
they proved a `Van Est theorem', stating that the map \eqref{eq:VEW} induces an isomorphism in cohomology in sufficiently low degrees
(depending on the connectivity properties of the fibers of the target map of $G$).

The Van Est map for groupoids, with its associated Van Est theorem, has a number of important applications. It arises in the context of integration problems for Poisson and Dirac manifolds  \cite{bur:intdir, cra:intpoi,cran:jacob} as well as for general Lie algebroids \cite{cra:dif,cra:intlie,tu:ktheory}. It is a tool in linearizing groupoid actions and Poisson structures \cite{cra:lin,wein:lin}, and
is related to the interplay between Cartan forms and Spencer operators \cite{cra:spenc,sal:thesis}. Finally, it enters the formulation of index theorems for foliations and more general groupoids \cite{con:ind,pfla:longind,pfla:geomind,pfla:tra}. 

The proof of a Van Est theorem in \cite{cra:dif} involves a certain double complex. In \cite{aba:wei}, this is enlarged to a triple complex. In this paper, we will show that this double/triple complex, in conjunction with the \emph {Perturbation Lemma} of homological 
algebra, may in fact be used to give a conceptual `explanation' for the van Est map itself.   The basic properties of the Van Est map follow rather easily from this approach. For example, one obtains a simple proof of the fact that the Van Est map  restricts to an algebra morphism on the normalized subcomplex, a fact first proven in \cite{meh:sup} via different techniques.

Let us briefly summarize this construction for the Van Est map \eqref{eq:weinxu}.
One begins by considering the principal $G$-bundles $\kappa_p\colon E_pG\to B_pG$, where $E_pG$ is the $p+1$-fold fiber product of $G$ with respect to the source map $\sz$.
The tangent bundle to the fibers of $\kappa_p$ defines a Lie algebroid $T_\F E_pG$.
The structure maps of the simplicial manifold $E_\bullet G$ lift to Lie algebroid morphisms; thus 
$T_\F E_\bullet G$ is a \emph{simplicial Lie algebroid}. One thus obtains a double complex, with bigraded summands 
$\sC^s(T_\F E_r G)$, and equipped with a Chevalley-Eilenberg differential $\d$ and a simplicial differential $\delta$. 
Let $\on{Tot}^\bullet \sC(T_\F EG)$ be the associated total complex. 
Pullback under the map to the base is a morphism of differential spaces
\begin{equation}\label{eq:1}
\kappa_\bullet^*\colon  C^\infty(B_\bullet G)\to \on{Tot}^\bullet \sC(T_\F EG).
\end{equation}
Similarly, the identification $T_\F E_0G=\sz^*A$ determines a pullback map $\sC(A)\to \sC(T_\F E_0G)$, 
which defines a morphism of differential spaces
\begin{equation}\label{eq:2}
\pi_0^*\colon \sC^\bullet(A)\to  \on{Tot}^\bullet \sC(T_\F EG).
\end{equation}
There is also a map $\iota_0^*\colon \on{Tot}^\bullet \sC(T_\F EG) \to \sC^\bullet(A)$ left inverse to $\pi_0^*$, 
defined using the inclusion  $A\hra T_\F E_0G$ with underlying map 
$M\hra E_0G$.  However, since this inclusion is not a Lie algebra morphism, the map $\iota_0^*$ is not a cochain map, 
in general. 

The simplicial manifold $E_\bullet G$ admits a canonical simplicial deformation retraction onto $M\subset E_\bullet G$. 
This determines a homotopy operator $h$ for the simplicial differential $\delta$ on the double complex $\sC^\bullet(T_\F E_\bullet G)$. 
We will prove:
\vskip.1in
\noindent{\bf Proposition.}
{\it The composition $\iota_0^*\circ (1+h\circ \d)^{-1}\colon \on{Tot}^\bullet \sC(T_\F EG) \to \sC^\bullet(A)$
is a cochain map, and is a homotopy inverse to $\pi_0^*$.}
\vskip.1in
This proposition is a fairly direct application of the \emph{Basic Perturbation Lemma} of homological algebra, due to 
Brown \cite{bro:twi} and Gugenheim \cite{gug:cha} (cf.~ Appendix \ref{app:perturb}).  
We will take the composition 
\begin{equation}\label{eq:def1}
 \on{VE}\colon \iota_0^*\circ (1+h\circ \d)^{-1}\circ \kappa^*\colon C^\infty(B_\bullet G)\to \sC^\bullet(A)
 \end{equation}
as a definition of the Van Est map. 
A more refined version of the Perturbation Lemma, due to Gugenheim-Lambe-Stasheff \cite{gug:per2}
(cf. Appendix \ref{app:perturb}) 
applies to cochain complexes with additional algebra structures. These conditions are not satisfied for the double complex $\sC^{\bullet}(T_\F E_\bullet G)$, but they do apply to the \emph{normalized subcomplex}. We thus recover the result of Weinstein-Xu \cite{wei:ext} that 
the Van Est map restricts to a ring homomorphism on the normalized subcomplex. 

The method generalizes to the Van Est map \eqref{eq:VEW} for the Bott-Shulman-Stasheff double complex. To this end, we will develop a new geometric description of the Weil algebra $\sW(A)$ of a Lie algebroid, as sections of a suitably defined \emph{Weil algebroid}. It 
may be regarded as a translation of the super-geometric approach of Weinstein and Mehta, and 
is of course equivalent to the description given by Abad-Crainic \cite{aba:wei}.
Working with the triple complex $\sW^{\bullet,\bullet}(T_\F E_\bullet G)$ we use the Perturbation Lemma to define 
the Van Est map:
\begin{equation}\label{eq:def2}
\on{VE}=\iota_0^*\circ (1+h\circ \d')^{-1}\circ \kappa^*\colon \Om^\bullet(B_\bullet G)\to \sW^{\bullet,\bullet}(A).
\end{equation} 
Here $\d'$ is the Chevalley-Eilenberg differential on $\sW^{\bullet,\bullet}(T_\F E_\bullet G)$. 
Again, we find that $\on{VE}$ restricts to an algebra morphism on a normalized cochains.  

Our final result is a direct formula for \eqref{eq:def2}, generalizing Equation \eqref{eq:expl1}.
Any section $X\in \Gamma(A)$ defines two kinds of 
contraction operators $\i_S(X)$ and $\i_K(X)$ on $\sW(A)$, of bidegrees $(-1,-1)$ and $(-1,0)$, respectively. (If 
$M=\pt$ so that $A=\g$ is a Lie algebra, we have $W^{p,q}(\g)=S^q\g^*\otimes \wedge^p\g^*$, and the two contraction operators 
are contractions on $S\g^*$ and $\wedge\g^*$, respectively.)
\vskip.1in
\noindent{\bf Theorem.}
{\it 
For $\phi\in \Om^q(B_pG)$, 
 $X_1,\ldots, X_p\in \Gamma(A)$, and any $n\le p$,
\[ \begin{split}
\lefteqn{\i(X_p)\cdots \i(X_{n+1})\i_S(X_n)\cdots \i_S(X_1)\on{VE}(\phi)}\\&=\iota^*\sum_{s\in\mf{S}_p} \eps(s) 
\L(X_{s(1)}^{1,\sharp})\cdots \L(X_{s(n)}^{n,\sharp})\i(X_{s(n+1)}^{n+1,\sharp})\cdots \i(X_{s(p)}^{p,\sharp})\phi.
\end{split}\]
Here $\iota\colon M\to B_pG$ is the inclusion as constant $p$-arrows, and $\eps(s)$ is $+1$ if the number of pairs 
$(i,j)$ with $1\le i< j\le n$ but $s(i)>s(j)$ is even, and $-1$ if that number is odd. 
}\vskip.1in

Our main motivation for developing our approach to  the Van Est map are integration problems for group-valued moment maps. This will be explained in a forthcoming paper. 
\vskip.2in

\noindent{\bf Acknowledgments.} We thank Marius Crainic, Rui Fernandes, Theodore Johnson-Freyd, and Xiang Tang for 
discussions and helpful comments. David Li-Bland was supported by an NSF Postdoctoral Fellowship, grant DMS 1204779. Eckhard Meinrenken was supported by an NSERC Discovery Grant.

\section{Lie groupoid and Lie algebroid cohomology}
We begin with a quick review of Lie groupoids, Lie algebroids, and the associated  cochain complexes. For more detailed information, see for example, Mackenzie \cite{mac:gen},  Moerdijk and Mr\v{c}un \cite{moe:fol} or Crainic-Fernandes \cite{cra:lect}.
\subsection{The De Rham complex of a simplicial manifold}\label{subsec:simpl1}
The basic definitions for simplicial manifolds are recalled in Appendix \ref{app:simplicial}. In short, a simplicial manifold 
is a contravariant functor $X\colon \on{Ord}\to \on{Man}$. Here $\on{Man}$ is the category of manifolds, with morphisms the 
smooth maps, and $\on{Ord}$ is the category of ordered sets $[p]=\{0,\ldots,p\}$ for $p=0,1,2,\ldots$, 
with morphisms the nondecreasing maps $[p']\to [p]$. One denotes 
$X_p=X([p])$. Of special significance are the \emph{face maps} $\partial_i\colon 
X_p\to X_{p-1}$ and \emph{degeneracy maps} $\eps_i\colon X_p\to X_{p+1}$, induced by the morphism 
$[p-1]\to [p]$ omitting $i$, respectively the morphism $[p+1]\to [p]$ repeating $i$.

The \emph{simplicial de Rham complex} of $X_\bullet$ is the double complex $\Om^\bullet(X_\bullet)$, 
with the simplicial differential 
\[ \delta=\sum_{i=0}^{p+1}(-1)^i \partial_i^*\colon \Om^q(X_p)\to \Om^q(X_{p+1}),\]
of bidegree $(1,0)$ and the second differential $\d=(-1)^p \d_{Rh}$ of bidegree 
$(0,1)$ where $\d_{Rh}$ is the de Rham differential. The two differentials commute in the graded sense, i.e. $\d\delta+\delta\d=0$, and both are graded derivations relative to the \emph{cup product}
\begin{equation}\label{eq:cupp}
 \phi\cup\phi'=(-1)^{p'q}\pr^* \phi\wedge (\pr')^*\phi'.
\end{equation}
Here $\pr\colon X_{p+p'}\to X_p$ and $\pr'\colon X_{p+p'}\to X_{p'}$ are the \emph{front face} and \emph{back face} projections, 
induced by the morphisms 
$[p]\to [p+p'],\ i\mapsto i$, respectively $[p']\to [p+p'], i\mapsto p+i$. If $S_\bullet\to X_\bullet$ is a simplicial vector bundle, with the property that the simplicial maps $S_\bullet$ are fiberwise isomorphisms, then the simplical differential $\delta$ extends to 
sections of $S_\bullet$ in an obvious way, and the cup-product generalizes to a product 
\[ \Om^q(X_p,S_p)\otimes \Om^{q'}(X_{p'},S_{p'}')\to \Om^{q+q'}(X_{p+p'},(S\otimes S')_{p+p'}).\]
Note however that only the simplicial differential $\delta$ is defined on $\Om^\bullet(X_\bullet,S_\bullet)$; 
the second differential is defined if $S_\bullet$ comes with a flat simplicial connection.

Occasionally it is better to work with the normalized subcomplex $\wt{\Om}^\bullet(X_\bullet,S_\bullet)$, consisting of forms 
that pull back to zero under all degeneracy maps. The normalized forms are a subalgebra with respect to the cup product.

Any manifold $M$ can be regarded as a simplicial manifold, by taking $M_p=M$ in all degrees and all simplicial structure maps to be the identity. The simplicial differential $\delta$ on $\Om^\bullet(M_\bullet)$ is given by the identity in odd degrees $p>0$ and zero otherwise. 


\subsection{Lie groupoids}\label{subsec:liegr}
Let $G\rra M$ be a Lie groupoid. The source and target maps are denoted by $\sz,\tz\colon G\to M$; they are submersions onto a 
submanifold $M\subseteq G$ of \emph{units}. Elements of $G$ are viewed as arrows
\[ m_0\xleftarrow{g} m_1\]
from $m_1=\sz(g)$ to $m_0=\tz(g)$. If $g$ and $g'$ are elements with $\sz(g)=\tz(g')$, then we write $gg'$ for their groupoid product. The groupoid inverse will be denoted by $g\mapsto g^{-1}$. Suppose $H\rra N$ is a second Lie groupoid. A smooth map 
$H\to G$ is called a \emph{morphism of Lie groupoids} if it restricts to a  map of units 
and intertwines all the structure maps for the Lie groupoids. It is depicted as a diagram
\begin{equation}\label{eq:morgr}
 \xymatrix{{H} \ar@<0.5ex>[r] \ar@<-0.5ex>[r]
\ar[d]_{\hat{\mu}} & {N}\ar[d]^{\mu}\\
{G } \ar@<0.5ex>[r] \ar@<-0.5ex>[r]& M
}
\end{equation}
If the map $(\hat{\mu},\sz)\colon H\to G\, _\sz\times_\mu N$ is a diffeomorphism, then
we say that \emph{$G$ acts on $N$ along $\mu$}. In this case, $G\ltimes N:=G\, _\sz\times_\mu N$
is called the \emph{action groupoid}, its target map
\[ \tz\colon G\ltimes N\to N,\ (g,n)\mapsto g.n=\tz(g,n)\]
is called the \emph{action map}, and the map $\mu\colon N\to M$ is the \emph{moment map} 
for the action. In particular, $G$  acts on its space $M$ of units; here 
$N=M$, with $\mu$ the identity map. A \emph{principal $G$-bundle} 
\begin{equation}\label{eq:principal}
 \xymatrix{{P} \ar[r]_{\kappa} \ar[d]_{\mu} & {B}\\
M&}
\end{equation}
is a manifold $P$ with a $G$-action along $\mu$, together with 
submersion $\kappa\colon P\to B$ such that $\kappa\circ \tz=\kappa\circ \sz$ as maps 
$G\ltimes P\to B$, and such that the map 
\begin{equation} \label{eq:gro} (\tz,\sz)\colon G\ltimes P\to P\times_B P
\end{equation}
is a diffeomorphism.

To define the cochain complex for a Lie groupoid $G\rra M$, let 
\[ B_pG=\{(g_1,\ldots,g_p)\in G^p|\ \sz(g_i)=\tz(g_{i+1}),\ i=1,\ldots,p-1\}\] 
be the manifold of \emph{$p$-arrows}
\begin{equation}\label{eq:parrow}
m_0 \xla{g_1} m_1 \xla{g_2} m_2 \xla{}\cdots \xla{g_p} m_p,
\end{equation} 
with base points $m_0,\ldots,m_p\in M$.  For $p=0$ we put $B_0G=M$.
Then $B_\bullet G$  is a simplicial manifold: the map 
$BG(f)\colon B_pG\to B_{p'}G$ defined by a nondecreasing map 
$f\colon [p']\to [p]$ takes the $p$-arrow \eqref{eq:parrow} to the $p'$-arrow 
\[ m_{f(0)} \xla{g_1'} m_{f(1)} \xla{g_2'} m_{f(2)} \xla{}\cdots \xla{g'_{p'}} m_{f(p')},\]
where $g_i'$ is 
obtained by composition of arrows (or insertion of trivial arrows). That is, 
$g_i'=g_{f(i)+1}\cdots g_{f(i+1)}$ for $f(i)<f(i+1)$, and $g_i'=m_i$ for $f(i)=f(i+1)$. 
In particular the degeneracy maps $\eps_i\colon B_pG\to B_{p+1}G,\ i=0,\ldots,p$ repeat the $i$-th base point, by inserting a trivial arrow, while  the face map
$\partial_i\colon B_pG\to B_{p-1}G,\ i=0,\ldots,p$ drops the $i$-th base point $m_i$:
\[ \partial_i(g_1,\ldots,g_p)=\begin{cases}
(g_2,\ldots,g_p) & i=0,\\
(g_1,\ldots,g_ig_{i+1},\ldots,g_p)&0<i<p,\\
(g_1,\ldots,g_{p-1}) &i=p.
\end{cases}
\]
For $p=1$ we have $\partial_0(g)=\sz(g),\ \ \partial_1(g)=\tz(g)$.  
%
The de Rham complex $\Om^\bullet(B_\bullet G)$ of this simplicial manifold is a bidifferential algebra, 
called the \emph{Bott-Shulman-Stasheff} complex, after \cite{bo:rh,shu:th}. 
A $\delta$-cocycle in $\Om^q(B_0G)=\Om^q(M)$ is (by definition) a $G$-invariant $q$-form on $M$, and a $\delta$-cocycle 
$\alpha\in \Om^q(B_1G)=\Om^q(G)$  is a multiplicative $q$-form on $G$, i.e. the pull-back under groupoid 
multiplication $\Mult\colon B_2G\to G$ equals the sum $\pr_1^*\alpha+\pr_2^*\alpha$.  

The differential algebra $\Om^0(B_\bullet G)=C^\infty(B_\bullet G)$ (with the simplicial differential $\delta$) is the 
complex of \emph{differentiable groupoid cochains}.
The inclusion of units $\iota\colon M\to G$, regarded as a groupoid morphism from $M\rra M$ to $G\rra M$, 
defines an injective morphism of simplicial manifolds
$M_p=B_pM\to  B_pG$, with image  the trivial $p$-arrows. The  \emph{complex of germs}  $\Om^\bullet(B_\bullet G)_M$
is defined to be the quotient of $\Om^\bullet(B_\bullet G) $ by the ideal of forms vanishing on some neighborhood of 
$M_p\subseteq B_pG$.
Similarly we define $C^\infty(B_\bullet G)_M$. Note that these are also defined for \emph{local} Lie groupoids.

For each of the complexes considered above, there are also the normalized subcomplexes. These will be denoted  
$\wt{C}^\infty(B_\bullet G), \wt{\Om}^\bullet(B_\bullet G)$,  and so on. 

\begin{examples}\label{ex:groupoid}
\begin{enumerate}
\item Given a manifold $M$, let $\on{Pair}(M)=M\times M\rra M$ be the pair groupoid,  with source map $\sz(m',m)=m$ and target map $\tz(m',m)=m'$. The inclusion of units is the diagonal embedding $M\hra M\times M$, and the groupoid multiplication 
reads as $(m_1',m_1)(m_2',m_2)=(m_1',m_2)$, defined 
whenever $m_1=m_2'$. In this example, any $p$-arrow is uniquely determined by its base points, and the map taking a $p$-arrow to its base points defines an isomorphism $B_\bullet(\on{Pair}(M))=M^{\bullet+1}$ as simplicial manifolds, where the simplicial structure on the right hand side comes from the identification of $M^{p+1}$ as the set of maps $[p]\to M$. Thus 
$C^\infty(B_\bullet\on{Pair}(M))=C^\infty(M^{p+1})$, with the differential 
given by the formula 
\[ (\delta f)(m_0,\ldots,m_{p+1})=\sum_{i=0}^{p+1}(-1)^i f(m_0,\ldots,\wh{m_i},\ldots,m_{p+1}).\]
This complex has trivial cohomology. However, the complex  
$\sC^p(\on{Pair}(M))_M=C^\infty(M^{p+1})_M$ of 
\emph{germs} of functions along the diagonal
$M\subseteq M^{p+1}$ is the \emph{Alexander-Spanier complex} \cite{spa:al}, which is known to compute the 
cohomology of $M$ with coefficients in $\R$. 
\item
More generally, given a foliation $\ca{F}$ on $M$, one defines a groupoid $\on{Pair}_{\ca{F}}(M)$, consisting of pairs of points in the same leaf. The complex $\sC^p(\on{Pair}_\F(M))_M$ may be seen as a foliated version of the Alexander-Spanier complex; 
a coefficient system is a bundle with a fiberwise flat connection. 
\item Let $K$ be a Lie group, acting on a manifold $M$, and let  $G=K\ltimes M$. 
Then $C^\infty(B_\bullet G)$ computes the group cohomology of $K$ with coefficients in the $K$-module $C^\infty(M)$.
\end{enumerate}
\end{examples}
Any morphism $\hat{f}\colon G_1\to G_2$ of Lie groupoids  (cf. \eqref{eq:morgr}), with underlying map $f\colon M_1\to M_2$,  
extends to a morphism of simplicial manifolds $f\colon B_\bullet G_1\to B_\bullet G_2$, giving rise to a morphism of 
bidifferential algebras $f^*\colon \Om^\bullet(B_\bullet G_2)\to \Om^\bullet(B_\bullet G_1)$, 
and hence of differential algebras $f^*\colon C^\infty(B_\bullet G_2)\to  C^\infty(B_\bullet G_1)$. 
For example, the canonical morphism 
$(\tz,\sz)\colon G\to \on{Pair}(M)$ defines a morphism 
of differential graded algebras $C^\infty(M^{\bullet+1})\to C^\infty(B_\bullet G)$.

\subsection{Lie algebroid cohomology}
A Lie algebroid  is a vector bundle $A\to M$ with a bundle map $\a\colon A\to TM$ (the \emph{anchor}) and a 
Lie bracket on the space of sections $\Gamma(A)$ satisfying
\[ [X_1,f X_2]=f[X_1,X_2]+(\a(X_1)f)\, X_2,\]
for all $X_1,X_2\in \Gamma(A)$ and $f\in C^\infty(M)$. Morphisms of Lie algebroids
\begin{equation}\label{eq:lamorphism}
\xymatrix{{\!\!\!\!\! B} \ar[r]
\ar[d]_{\hat{\mu}}& {N}\ar[d]^\mu\\
{A\ } \ar[r]& M}
\end{equation}
are vector bundle maps such that the differential $T\mu\colon TN\to TM $ intertwines the anchor maps, and with a certain compatibility condition\footnote{If $B\subseteq A$ is a subbundle along a submanifold
$N\subseteq M$, the condition is that whenever $X_1,X_2\in \Gamma(A)$ extend sections $Y_1,Y_2\in\Gamma(B)$, 
then $[X_1,X_2]$ extends $[Y_1,Y_2]$.  The general 
case may be reduced to this case, by replacing the vector bundle map $\hat{\mu}$ by the inclusion $B\to A\times B$ of 
the graph of $\hat{\mu}$. (Cf.~ \cite{lib:cou}.)}
for the Lie brackets on sections, due to Higgins-Mackenzie \cite{hig:alg,mac:gen}. Such a morphism is called 
an \emph{action of $A$ on $N$ along $\mu$} if  the resulting map $B\to \mu^*A$ is an isomorphism; in this case 
$B$ is called the \emph{action Lie algebroid} and is denoted $A\ltimes N$. Given an $A$-action, the composition of 
$\mu^*\colon \Gamma(A)\to \Gamma(A\ltimes N)$ with the anchor map for $A\ltimes N$ defines a 
Lie algebra morphism $\Gamma(A)\to \Gamma(TN),\ X\mapsto X_N$, such that $X_N\sim_\mu \a(X)$.

The \emph{Chevalley-Eilenberg complex} of $A$ 
is the graded 
differential algebra $\sC^\bullet(A)=\Gamma(\wedge^\bullet A^*)$, with product the wedge product, and with the differential 
%
$\d_{\on{CE}}\colon \sC^\bullet(A)\to \sC^{\bullet+1}(A)$
%
given as 
%
 \begin{equation}\label{eq:CEC}
  \begin{split}
 (\d_{CE}\phi)(X_0,\ldots,X_p)=& \sum_{i=0}^p (-1)^i \a(X_i)\phi(X_0,\ldots,\wh{X_i},\ldots,X_p)\\
& +\sum_{i<j} (-1)^{i+j}\phi([X_i,X_j],X_0,\ldots,\wh{X_i},\ldots,\wh{X_j},\ldots,X_p).\end{split}\end{equation}
%
%
%
\begin{examples}
\begin{enumerate}
\item  Given an action of a Lie algebra $\k$ on $M$, let $A=\k\ltimes M$ be the action Lie algebroid. 
Then $\sC^\bullet(A)=
 C^\infty(M)\otimes \wedge^\bullet \k^*$ is the Chevalley-Eilenberg complex of $\k$ with coefficients in $C^\infty(M)$.
\item Given a foliation $\ca{F}$ on $M$, let $A=T_{\ca{F}}M\subseteq TM$ be the tangent bundle to the foliation. Then 
$\sC^\bullet(A)= \Om^\bullet_{\ca{F}}(M)$ is the de foliated Rham complex 
(i.e., the quotient of $\Om(M)$ by forms whose pull-back to leaves are zero). 
\item Given an embedded hypersurface $N\subseteq M$, there is a Lie algebroid $A=T_NM$ whose sections are the vector fields tangent to 
$N$. (For manifolds with boundary, this is the starting point for Melrose's $b$-calculus \cite{mel:ati}.) The corresponding complex $\sC^\bullet(A)=\Om_N^\bullet(M)$ may be regarded as a space of forms on $M\backslash N$ developing a `logarithmic' singularity along $N$.  More generally, given a Lie algebroid $P\to M$ and a Lie subalgebroid
$Q\to N$ along a hypersurface, there is a Lie algebroid $A=[P:Q]$ whose sections are the sections $\sigma\in \Gamma(P)$ with the property 
$\sigma|_N\in\Gamma(Q)$. See Gualtieri-Li \cite{gua:sym}. 
\item Given a Poisson structure $\pi$ on $M$, the cotangent bundle $A=T^*M$ acquires the structure of a Lie algebroid
with anchor map $\a=\pi^\sharp \colon T^*M\to TM$, and with bracket the \emph{Koszul bracket}. The resulting differential on the algebra $\sC^\bullet(A)=\mf{X}^\bullet(M)$  of multi-vector fields is the Koszul differential $\d_\pi=[\pi,\cdot]$; 
its cohomology is the Poisson cohomology of $M$.
\end{enumerate}
\end{examples}

Any morphism of Lie algebroids $A_1\to A_2$, with underlying map $f\colon M_1\to M_2$, 
gives rise to a morphism of differential algebras $f^*\colon \sC^\bullet(A_2)\to \sC^\bullet(A_1)$. 
As a special case, the anchor map $\a\colon A\to TM$ of a Lie algebroid gives a morphism 
\[ \a^*\colon \Om^\bullet(M)=\sC^\bullet(TM)\to \sC^\bullet(A).\]

The infinitesimal counterpart to the bigraded algebra $\Om(BG)$ for a Lie groupoid is the \emph{Weil algebra} $\sW(A)$. 
A geometric model for $\sW(A)$ will be described in Section \ref{sec:weil}.

\subsection{The Lie functor}\label{subsec:liefun}
For any Lie groupoid $G\rra M$, the normal bundle 
\[ \on{Lie}(G)=\nu(M,G)\]
of $M$ in $G$ has the structure of a Lie  algebroid, with anchor map 
$\a\colon  \on{Lie}(G)\to  TM$ induced by the difference $T\tz-T\sz\colon TG\to TM$, and with the 
Lie bracket on sections defined by the identification 
\[\Gamma(\on{Lie}(G))=\on{Lie}(\Gamma(G))\]
with the Lie algebra of the infinite-dimensional group of bisections $\Gamma(G)$. 
Equivalently, the Lie bracket comes from the identification of sections $X\in \Gamma(\on{Lie}(G))$ with the Lie algebra of 
left-invariant vector 
fields $X^L\in\mf{X}(G)$ (tangent to $\tz$-fibers).
The definition of $\on{Lie}(G)$ also makes sense for
\emph{local} Lie groupoids, and it is known that any Lie algebroid $A$ arises in this way. The precise obstructions for integration to a \emph{global} Lie groupoid were determined by Crainic-Fernandes \cite{cra:exi}. 

Any  $G$-action on a manifold $N$ gives rise to a $\on{Lie}(G)$-action, 
with the action Lie algebroid $\on{Lie}(G)\ltimes N=\on{Lie}(G\ltimes N)$.   For a principal $G$-bundle $P$ as in \eqref{eq:principal}, the action Lie algebroid 
has an injective anchor map, and identifies $\on{Lie}(G)\ltimes P$ with the subbundle $\ker(T\kappa)\subseteq TP$ where 
$\kappa\colon P\to B$ is projection to the base. We hence have identifications 
\[ \mu^*\on{Lie}(G)\cong \on{Lie}(G)\ltimes P \cong \ker(T\kappa),\]
and a Lie algebroid morphism from $\ker(T\kappa)$ to $\on{Lie}(G)$. These remarks apply in particular to 
the action of $G$ on itself along $\tz$, given by multiplication from the left, as well as to the action along $\sz$, given by multiplication from the right. It identifies $\tz^*\on{Lie}(G)=\ker(T\sz)$ and $\sz^*\on{Lie}(G)=\ker(T\tz)$. On the level of sections, 
$\tz^*X=-X^R$ are the generating vector fields for the left action, while 
$\sz^*X=X^L$ are the generating vector fields for the right action.  These vector fields satisfy the commutation relations
\[ [X_1^L,X_2^L]=[X_1,X_2]^L,\ \ \ [X_1^R,X_2^R]=-[X_1,X_2]^R,\ \ \ [X_1^L,X_2^R]=0.\]
The differences $X^L-X^R$ are the generating vector fields for the conjugation action of the group $\Gamma(G)$ on $G$. 
(There is no conjugation action of $G$ on itself unless $M=\pt$.) They are tangent to $M$, and restrict to the vector field $\a(X)$.

\section{The Van Est map $C^\infty(BG)\to \sC(A)$}\label{sec:VEd}
In his proof of the Van Est theorem for Lie groupoids \cite{cra:dif}, Crainic introduced a double complex  with cochain maps from both the 
Lie algebroid complex and the Lie groupoid complex. In this section, we will use this double complex to define the Van Est map itself.

\subsection{The simplicial principal bundle $EG$}\label{subsec:simp}
For any Lie groupoid $G\rra M$ let 
\[ E_pG=\{(a_0,\ldots,a_p)\in G^{p+1},\ \sz(a_0)=\ldots=\sz(a_p)\}.\]
(cf.~ \cite[page 53]{aba:th} and Appendix \ref{app:simplicial}), and let $\pi_p\colon E_pG\to M$ 
be the common source map, $\pi_p(a_0,\ldots,a_p)=\sz(a_0)$. The space $E_pG$ has the structure of a principal 
$G$-bundle 
\begin{equation}\label{eq:thereare}
 \xymatrix{{E_pG} \ar[r]_{\kappa_p} \ar[d]_{\pi_p} & {B_pG}\\
M&}
\end{equation}
for the $G$-action $g.(a_0,\ldots,a_p)=(a_0\,g^{-1},\ldots,a_p\,g^{-1})$ along $\pi_p$, and with the 
quotient map $\kappa_p(a_0,\ldots,a_p)=(a_0a_1^{-1},\ldots,a_{p-1}a_p^{-1})$. The collection of the 
spaces defines a simplicial principal $G$-bundle $E_\bullet G\to B_\bullet G$: Regarding
$E_pG$ as maps $[p]\to G$ whose composition with the source map is constant, 
the structure map $E_pG\to E_{p'}G$ for a nondecreasing map
$f\colon [p']\to [p]$ is given by composition. In particular, 
the face maps $\partial_i\colon E_pG\to E_{p-1}G$ drop the $i$-th entry, 
while the degeneracy maps $\eps_i\colon E_pG\to E_{p+1}G$ repeat the $i$-th entry. 
Any groupoid morphism $G_1\to G_2$ defines a morphism of simplicial principal bundles $E_\bullet G_1\to E_\bullet G_2$. 

\begin{remark}
The simplicial manifold $E_\bullet G$ may be equivalently defined as $E_pG=B_p (G\ltimes G)$, 
where $G\ltimes G$ is the action groupoid for the action $g.a=ag^{-1}$. Here 
$\kappa_p$ is obtained by applying the functor $B_\bullet$ to the groupoid morphism 
$G\ltimes G\to G$.  See \cite[Definition 3.2.4]{aba:th}. 
\end{remark}

\subsection{Retraction of $EG$ onto $M$}
For the trivial groupoid $M\rra M$ we have
$E_pM=B_pM=M$ in all degrees. The inclusion $\iota\colon M\to G$ as units is a groupoid morphism, defining a  
simplicial map 
\[ \iota_p\colon M_p\to E_pG,\ \ m\mapsto (m,\cdots,m)\]
with $\pi_p\circ \iota_p=\on{id}_M$.
In Appendix \ref{app:A-simhom}, we show that there is a canonical simplicial deformation retraction 
from $E_\bullet G$ onto the submanifold $M$.  In turn, this defines a homotopy operator for 
the de Rham complex of $E_\bullet G$. For $0\le i\le p$ let
\begin{equation}\label{eq:fullform2}h_{p,i}\colon E_pG\to E_{p+1}G,\ \ 
(a_0,\ldots,a_p)\mapsto (a_0,\ldots,a_i,m,\ldots,m),
\end{equation} 
with $p+1-i$ copies of $m=\sz(a_0)=\ldots=\sz(a_p)$. 
The homotopy operator  is given by 
\begin{equation}\label{eq:fullform1} h=\sum_{i=0}^{p-1} (-1)^{i+1} (h_{p-1,i})^*\colon \Om^q(E_pG)\to \Om^q(E_{p-1}G).
\end{equation}
Thus $h\delta+\delta h=\on{id}-\pi_\bullet^*\iota_\bullet^*$.
%
%
For any morphism of Lie groupoids $f\colon G_1\to G_2$, the pullback map $f^*\colon \Om(E_\bullet G_2)\to \Om(E_\bullet G_1)$ intertwines the homotopy operators. 
\begin{example}\label{ex:triv}
In particular, the inclusion $\iota\colon M\to G$, viewed as a morphism from 
$M\rra M$ to $G\rra  M$, satisfies $h\circ \iota_\bullet^*=\iota_\bullet^*\circ h$.
Note that the simplicial complex $(\Om(M_\bullet),\delta)$  is simply 
\[ \Om(M)\xrightarrow{0}\Om(M)\xrightarrow{\on{id}}\Om(M)\xrightarrow{0}\Om(M)\cdots;\]
i.e., $\delta$ is the identity in odd degrees $p>0$ and zero otherwise. The homotopy operator 
$h$ on this complex restricts to the identity in odd degrees $p>0$ and zero otherwise.

There is also a homotopy operator $k$ for the inclusion of $\Om(M)\hra \Om(M_\bullet)$ as the degree $0$ piece, 
with homotopy inverse the projection. The operator $k$ is the identity in even degrees $p>0$ and zero otherwise. 
\end{example}
\begin{proposition}\label{prop:homproperties}
The homotopy operator $h\colon \Om^\bullet(E_\bullet G)\to \Om^\bullet(E_{\bullet-1}G)$ has the following 
additional properties:
\begin{enumerate}
\item\label{it:1} $h\circ h=0$. 
\item\label{it:mod}
 $h$ is an $\Om(M)$-module morphism, in the sense that 
\[ h(\alpha\wedge \pi_p^*\beta)=h\alpha\wedge\pi_{p-1}^*\beta\]
for all $\alpha\in \Om(E_pG)$ and $\beta\in \Om(M)$.
\item\label{it:3} The homotopy operator is an $R$-twisted derivation, for the algebra morphism
$R=\pi_\bullet^*\circ \iota_\bullet^*$. That is, 
\[ h(\alpha\cup\alpha')=h\alpha\cup R\alpha'+(-1)^{|\alpha|}\alpha\cup h\alpha'\]
for $\alpha\in{\Om}^q(E_pG)$ and $\alpha'\in {\Omega}^{q'}(E_{p'}G)$. 
\item\label{it:4} The homotopy operator preserves the normalized subcomplex $\wt{\Om}(E_\bullet G)$. The composition 
$\iota_\bullet^*\circ h$ vanishes on the normalized subcomplex. 
\end{enumerate}
\end{proposition}
\begin{proof}
Part \eqref{it:1} is obtained by duality to its homological counterpart (
Proposition \ref{prop:fol}). Part \eqref{it:mod} follows since
$\pi_p\circ h_{p,i}=\pi_{p-1}$, whence 
$h_{p,i}^*(\alpha\wedge \pi_p^*\beta)=h_{p,i}^*\alpha\wedge (\pi_{p-1})^*\beta$. For Part \eqref{it:3}, note that 
\[ (h_{p+p'-1,i})^*(\alpha\cup \alpha')=\begin{cases} (h_{p-1,i})^*\alpha\cup R\alpha' & i\le p-1\\
(-1)^q\alpha\cup (h_{p'-1,i-(p-1)})^*\alpha' & i>p-1\end{cases}
\]
where the sign comes from the sign convention for the cup product. Taking sum of these terms 
from $i=0$ to $i=p+p'-1$, with alternating sign $(-1)^{i+1}$, the sum from $i=0$ to  $i=p-1$ 
gives $h\alpha\cup R\alpha'$, while the sum from $i=p$ to $i=p+p'-1$ gives $(-1)^{p+q}\alpha\cup h\alpha'$.
As for Part \eqref{it:4}, it is clear that $h$ preserves the normalized subcomplex $\wt{\Om}(E_\bullet G)$. 
The composition $\iota_{p-1}^*\circ h=h\circ \iota_p^*$ vanishes on $\wt{\Om}(E_pG)$ with $p>0$ since $\iota_p^*$
vanishes there, and for $p=0$ since $h$ vanishes there. 
\end{proof}

\subsection{Van Est Double complex}
Let $T_\F E_pG=\ker(T\kappa_p)$ be the tangent bundle to the foliation $\F$ defined by the fibers of the 
principal bundle $\kappa_p\colon E_pG\to B_pG$. As for any principal groupoid bundle (see Section \ref{subsec:liefun}), we have isomorphisms 
\[ \pi_p^*A\cong A\ltimes E_pG\cong T_\F E_pG,\]
and the resulting map $A\ltimes E_pG\to A$ is a Lie algebroid morphism. In fact, 
$T_\F E_\bullet G$ is a simplicial Lie algebroid, and the map to $A$ is a morphism of 
simplicial Lie algebroids
\[ \hat{\pi}_\bullet\colon T_\F E_\bullet G\to A_\bullet,\]
where $A_p=A$ for all $p$ (with all simplicial structure maps the identity). 
Following \cite{aba:wei,cra:dif} we define the \emph{Van Est double complex} 
\begin{equation}\label{eq:ved}
 \sC^{r,s}(T_\F EG):=\sC^s(T_\F E_rG),
 \end{equation}
with the simplicial differential $\delta$ of bidegree $(1,0)$ and the differential 
$\d=(-1)^r \d_{CE}$ of bidegree $(0,1)$; the extra sign is introduced so that $[\d,\delta]=\d \delta+\delta\d=0$. 
The space $\sC^\bullet(T_\F E_\bullet G)$ is a bidifferential algebra for the cup product
\begin{equation}\label{eq:cupo}
 \sC^{s}(T_\F E_rG)\otimes \sC^{s'}(T_\F E_{r'}G)\to \sC^{s+s'}(T_\F E_{r+r'}G)\end{equation}
defined by $\phi\cup\phi'=(-1)^{r's}\pr^*\phi\ (\pr')^*\phi'$, with the front face projection 
$\pr\colon E_{r+r'}G\to E_rG$ and the back face projection $\pr'\colon E_{r+r'}G\to E_{r'}G$. 
\begin{remark}
For any fixed $r$, the complex 
$\sC^{\bullet}(T_\F E_rG)$ with differential $\d_{CE}$ is the foliated de Rham complex $\Om^\bullet_{\ca{F}}(E_rG)$ 
for the fibration $\kappa_r\colon E_rG\to B_rG$.  
\end{remark}

Consider again the simplicial Lie algebroid $A_\bullet$. The corresponding bidifferential algebra has 
summands $\sC^s(A_r)=\sC^s(A)$; the simplicial differential $\delta$ vanishes on this summand when $r$ is even and is the identity map if $r$ is odd, while
$\d=(-1)^r\d_{CE}$ as before.  The map $\pi_r\colon E_rG\to M$ lifts to a morphism of simplicial Lie algebroids, 
$T_\F E_rG\to A$. 
Regard $C^\infty(B_\bullet G)$ as a bidifferential algebra concentrated in bidegrees $(\bullet,0)$.
We obtain a diagram
\[ \xymatrix{ {\sC^{\bullet}(T_\F E_\bullet G)} & {C^\infty(B_\bullet G)}\ar[l]_{\ \ \kappa_\bullet^*} \\
{\sC^{\bullet}(A_\bullet)} \ar[u]^{\pi_\bullet^*} &}
\]
where both maps are morphisms of bidifferential algebras.

\subsection{Definition of the Van Est map}\label{subsec:vanest}
The vector bundle morphism 
\begin{equation}\label{eq:inclusion} 
\xymatrix{ 
{T_\F E_rG}\ar[r] & E_rG\\
{A_r}\ar[u]^{}\ar[r] & M_r \ar[u]_{\iota_r}
}
\end{equation}
defines a morphism of bigraded 
spaces
\[ \iota_\bullet^*\colon \sC^{\bullet}(T_\F E_\bullet G)\to \sC^{\bullet}(A_\bullet),\]
which is right inverse to $\pi_\bullet^*$. This morphism intertwines $\delta$, but usually not $\d$ since 
\eqref{eq:inclusion}
is not a Lie algebroid morphism, in general.   Homological perturbation theory (Appendix \ref{app:perturb}) modifies this map, in such a way that it intertwines the total differentials $\d+\delta$. 

The construction uses a homotopy operator for the differential $\delta$. For any fixed $s$, the complex $\sC^{s}(T_\F E_\bullet G)$ is the simplicial complex of $E_\bullet G$ with coefficients in the simplicial vector bundle 
\[ \wedge^s T_\F^*E_\bullet G\cong \pi_\bullet^*\, \wedge^sA^*.\]
Since the maps $h_{r,i}\colon E_rG\to E_{r+1}G$ lift to vector bundle morphisms $T_\F E_rG=\pi_r^*A\to T_\F E_{r+1}G
=\pi_{r+1}^*A$,  we have a well-defined homotopy operator   with respect to the simplicial differential $\delta$
given once again by the formula \eqref{eq:fullform1}, 
$h=\sum_i (-1)^i (h_{r-1,i})^*$. On the dense subspace  
\begin{equation}\label{eq:incl}
 C^\infty(E_\bullet G)\otimes_{C^\infty(M)}\sC^s(A)\subseteq \sC^{s}(T_\F E_\bullet G),
\end{equation}
it acts as the given homotopy operator on $C^\infty(E_\bullet G)$, tensored with the identity operator on 
$\sC^s(A)$.

Both $\d \circ h$ and $h\circ  \d$ are operators of bidegree $(-1,1)$ on $\sC^{\bullet}(T_\F E_\bullet G)$.  Hence they 
are nilpotent operators of total degree $0$, and  $1+\d \circ h$ and $1+h \circ \d$ are invertible operators of total degree zero. 
The \emph{Perturbation Lemma} of homological algebra 
(cf.~ Lemma \ref{lem:B} in Appendix \ref{app:perturb}) gives the following statement:
\begin{lemma}
The map
\[ \iota_\bullet^*\circ (1+\d\circ h)^{-1}\colon \on{Tot}^\bullet \sC(T_\F EG)\to \on{Tot}^\bullet\sC(A)\]
is a cochain map for the total differential $\d+\delta$, and is a homotopy equivalence, with 
homotopy inverse $(1+h\circ \d)^{-1}\circ \pi^*_\bullet$. 
\end{lemma}

Here $\on{Tot}^\bullet\sC(A)$ indicates the total complex of the double complex $\sC^\bullet(A_\bullet)$.
The inclusion $\sC^\bullet(A)\equiv \sC^{\bullet}(A_0)\subseteq\on{Tot}^\bullet \sC(A)$ 
is also a homotopy equivalence, with homotopy inverse the projection. (The corresponding homotopy operator 
$k\colon  \sC^{s}(A_r)\to \sC^{s}(A_{r-1})$ is the identity map for $r>0$ even, and $0$ otherwise - cf. Example \ref{ex:triv}.) 
By composing the two homotopy equivalences, and observing that 
$(1+h\circ \d)^{-1}\circ \pi^*_0=\pi_0^*$ (for degree reasons), we obtain:
\begin{proposition}\label{prop:hmk}
The map 
\[ \iota_0^*\circ  (1+\d\circ h)^{-1}\colon  \on{Tot}^\bullet\sC(T_\F EG)\to \sC^\bullet(A)\]
intertwines the total differential $\d+\delta$ with the Chevalley-Eilenberg differential. It is a homotopy equivalence, with homotopy inverse 
the map $\pi_0^*$.
\end{proposition}
Here $\iota_0^*$ is regarded as a map on the full double complex, given by $0$ on 
$\sC^s(T_\F E_r G)$ with $r>0$, and similarly $\pi_0^*$ is viewed as a map into the full double complex. 
Composing with the cochain map 
\[ \kappa_\bullet^*\colon C^\infty(B_\bullet G)\to \sC^0(T_\F E_\bullet G)\subseteq \on{Tot}^\bullet\sC(T_\F EG)\]
we arrive at the following definition:
\begin{definition}
Let $G\rra M$ be a Lie groupoid, with Lie algebroid $A=\on{Lie}(G)$. 
The composition
\begin{equation}\label{eq:composition}
 \on{VE}=\iota_0^*\circ (1+\d \circ h)^{-1}\circ \kappa_\bullet^*\colon 
C^\infty(B_\bullet G)\to \sC^\bullet(A) 
 \end{equation}
 is called  the \emph{Van Est map}.
\end{definition}
By construction, $\on{VE}$ is a cochain map. We will verify in Section \ref{sec:wexu} that it coincides with Weinstein-Xu's definition of the Van Est map. 
\begin{remarks}
\begin{enumerate}
\item The map $\on{VE}$  is functorial:
Let $G_1\to G_2$ be a morphism of Lie groupoids,  
and let $A_1\to A_2$ be the corresponding morphism of Lie algebroids. 
From the construction of the Van Est map, it is immediate that the 
following diagram commutes:
\[ \begin{CD} 
{C^\infty(B_\bullet G_2)}
@>>> {C^\infty(B_\bullet G_1)} \\@V{\on{VE}}VV @VV{\on{VE}}V\\
\sC^\bullet(A_2) @>>> \sC^\bullet(A_1)
\end{CD}
\]
\item 
Since $\d\circ h$ has bidegree $(-1,1)$, the Van Est map has the following `zig-zag' form on elements $\phi\in C^\infty(B_pG)$:
\[
\on{VE}(\phi)=(-1)^p \iota_0^* \circ (\d\circ h)^p \circ \kappa_p^* \phi.
\]
\item 
The Van Est map can also be written 
\[ \on{VE}=\iota_0^*\circ (1+[h,\d])^{-1}\circ \kappa^*\]
because $(1+[h,\d])^{-1}=(1+\d h)^{-1}+\sum_{j=1}^\infty (h \d)^j$ and
 $\d \circ  \kappa^*=0$. This alternative form turns out to be easier to work with, since $[h,\d]$ is closer 
 to being a derivation.  
\item 
For $G$ a possibly local Lie groupoid, we can consider the differential algebra of germs $C^\infty(B_\bullet G)_M$. Using 
the double complex  
$\sC^{\bullet}(T_\F E_\bullet G)_M$ of germs along $M\subseteq EG$ one obtains a Van Est map 
\[ \on{VE}_M\colon C^\infty(B_\bullet G)_M\to \sC^\bullet(A).\]
For a global Lie groupoid, the map $\on{VE}$ factors as the natural projection $C^\infty(B_\bullet G)\to C^\infty(B_\bullet G)_M$ 
followed by $\on{VE}_M$.
\end{enumerate}
\end{remarks}

The Van Est map on the full complex of groupoid cochains
fails to be an algebra homomorphism, in general. However, 
it does respect products on the normalized subcomplex \cite[Proposition 6.2.3]{meh:sup}. 
\begin{theorem}\label{th:multi}
The Van Est map for the trivial $G$-module restricts to an algebra morphism 
$\on{VE}\colon \wt{C}^\infty(B_\bullet G)\to \sC^\bullet(A)$
on the normalized subcomplex. 
\end{theorem}
\begin{proof}
The compatibility of the homological perturbation theory with algebra structures is addressed in the work of 
Gugenheim-Lambe-Stasheff \cite{gug:per2} (see Appendix \ref{app:perturb}, Lemma \ref{lem:B}). To apply their result,  
we need to verify the  \emph{side conditions}
$h\circ h=0,\ \ \iota^*\circ h=0$ as well as the $R_0:=\pi_0^*\iota_0^*$-derivation property.
But these follow from Proposition \ref{prop:homproperties}, and since 
\[  \wt{C}^\infty(E_rG)\otimes_{C^\infty(M)}\sC^\bullet(A)\subseteq \wt{\sC}^{\bullet}(T_\F E_r G)\]
is a dense subspace.
\end{proof}

\subsection{Coefficients}
The theory described above admits a straightforward generalization to the case with coefficients. A module over a Lie algebroid  
$A\to M$ is a vector bundle $\mu\colon S\to M$, equipped with a \emph{linear} $A$-action. The linearity condition is the requirement that  
$A\ltimes S\to S$ is a \emph{VB-algebroid} \cite{gra:vba,mac:gen} over $A\to M$ (also called \emph{LA-vector bundle}). 
Equivalently,  $S$ comes equipped with a  flat $A$-connection $\nabla\colon \Gamma(S)\to \Gamma(A^*\otimes S)$, i.e. 
\[ \nabla_X(f\sigma)=f\nabla_X\sigma+(\a(X)f)\sigma,\ \ \ [\nabla_X,\nabla_Y]=\nabla_{[X,Y]}.\]
(For example, if $\ca{F}$ is a foliation on $M$, then a $T_\F M$-module is given by a vector bundle with a flat connection in 
the direction of the fibers.) One obtains a complex $\sC^\bullet(A,S)=\Gamma(\wedge^\bullet A^*\otimes S)$, with a differential $\d_{CE}$ 
given by a similar formula \eqref{eq:CEC} as before, replacing $\a(X)$ with $\nabla_X$.  Given another $A$-module $S'$, the 
wedge product  
\[ \sC^\bullet(A;S)\otimes \sC^\bullet(A;S')\to \sC^\bullet(A;S\otimes S'),\ \phi\otimes\phi'\mapsto \phi\wedge\phi'\]
is a morphism of differential spaces. 

Similarly, a module over a Lie groupoid $G\rra M$ is a vector bundle $\mu\colon S\to M$ with a \emph{linear} $G$-action along $\mu$, 
i.e. the action groupoid $G\ltimes S\rra S$ is a 
\emph{VB-groupoid} in the sense of Pradines \cite{mac:gen,pra:rem,gra:vb}. Equivalently, for 
any groupoid element $g\in G$ the map $S_{\sz(g)}\to S_{\tz(g)},\ v\mapsto g.v$
is linear. There is a similar definition of modules for local Lie groupoids. Any $G$-module becomes a $\on{Lie}(G)$-module for the 
infinitesimal action. 

Given a $G$-module $S\to M$, we obtain a simplicial vector bundle 
$B_\bullet(G\ltimes S)\to B_\bullet G$. We obtain a cochain complex of sections of this bundle, with the simplicial differential defined as before. (One can also consider the bigraded space of bundle-valued differential 
forms, but in order to define a second differential on this space one needs a $G$-invariant flat connection  on $S$; see
Section \ref{subsec:simpl1}.) 
\begin{remark}
The fiber of $B_p(G\ltimes S)$ at a $p$-arrow $(g_1,\ldots,g_p)$ 
(cf. \ref{eq:parrow}) consists of 
tuples $(v_0,\ldots,v_p)$  of elements $v_i\in S_{m_i}$, with $v_{i-1}=g_i.v_i$.   Any such tuple is determined by the element 
$v_p$; hence $ B_p(G\ltimes S)\cong B_pG\times_M S$. 
\end{remark}
Consider the $G$-equivariant simplicial vector bundle $E_\bullet(G\ltimes S)$.  
The common source map for elements of this bundle defines a vector bundle map onto $S$, with underlying 
map $\pi_p$. Thus $E_p(G\ltimes S)=\pi_p^* S$. 
On the other hand, the total space of $E_p(G\ltimes S)$ is a principal bundle over $B_p(G\ltimes S)$, and the quotient map 
identifies $E_p(G\ltimes S)=\kappa_p^* B_p(G\ltimes S)$. 

The vector bundle $\pi_p^* S=E_p(G\ltimes S)$ is a $T_\F E_pG$-module, hence a double complex 
$\sC^\bullet(T_\F E_\bullet G,\pi_\bullet^* S)$ is defined. By repeating the argument from the last section, we 
use the homotopy operator on this double complex to define the Van Est map 
\[ \on{VE}=\iota_0^*\circ (1+h\circ\d)^{-1}\circ \kappa^*\colon \Gamma(B_\bullet (G\ltimes S))\to \sC^\bullet(A,S).\]
Given two $G$-modules $S,S'\to M$ one obtains a commutative diagram for the normalized subcomplexes 
\[ \begin{CD} 
{\wt{\Gamma}(B_\bullet (G\ltimes S))  \otimes \wt{\Gamma}(B_\bullet (G\ltimes S'))} @>{\cup}>> \wt{\Gamma}(B_\bullet (G\ltimes (S\otimes S'))) \\@V{\on{VE}\otimes \on{VE}}VV @VV{\on{VE}}V\\
\sC^\bullet(A;S)\otimes \sC^\bullet(A;S') @>{\cup}>> \sC^\bullet(A;S\otimes S')
\end{CD}
\]
The argument is essentially the same as in the case of trivial coefficients, see Remark \ref{rem:gug}.

\section{The Weil algebroid}\label{sec:weil}
As discussed in Section \ref{subsec:liegr}, the groupoid cochain complex $C^\infty(B_\bullet G)=C^\infty(B_\bullet G)$ extends to the Bott-Shulman-Stasheff double complex $\Om^\bullet(B_\bullet G)$. To extend the Van Est map to this double 
complex, we need a description of the infinitesimal counterpart $\sW^{\bullet,\bullet}(A)$, the \emph{Weil algebra} of
a Lie algebroid $A$.  The definition of this algebra, and a construction of the corresponding Van Est map,  was given 
by Mehta \cite{meh:sup} and Weinstein (unpublished notes) in terms of super geometry, and by Abad-Crainic \cite{aba:wei} using their theory of \emph{representations up to homotopy}. The geometric model given below, as sections of a `Weil algebroid', may be seen as a
translation of Mehta-Weinstein's definition into ordinary differential geometry. 

\subsection{Koszul algebroids}\label{subsec:koszul}
Let $A\to M$ be any vector bundle. We will define a `Koszul algebroid' $W(A)$ as a module of K\"ahler differentials for the bundle of graded algebras $\wedge A^*$. 
Consider $\wedge A^*$ as a bundle of commutative graded algebras, and let  
\begin{equation}
 \mf{der}(\wedge A^*)=\bigoplus_{i\in \Z}\mf{der}^i(\wedge A^*)
\end{equation}
be the graded  vector bundle over $M$ whose sections are the
graded derivations of $\Gamma(\wedge A^*)$.  Its fiber 
$\mf{der}(\wedge A^*)_m$ at $m\in M$ is the space of graded derivations $D_m\colon \Gamma(\wedge A^*)\to \wedge A^*_m$ 
of the graded $\Gamma(\wedge A^*)$-module $\wedge A^*_m$. Since $\wedge A^*$
is graded commutative, the bundle $\mf{der}(\wedge A^*)$ is a graded $\wedge A^*$-module.

\begin{proposition}
There is a short exact sequence of graded $\wedge A^*$-modules
\begin{equation}\label{eq:exact2} 
0\to \wedge A^*\otimes A\to \mf{der}(\wedge A^*)\to \wedge A^*\otimes TM\to 0.
\end{equation}
Here the second factor in $\wedge A^*\otimes A$ has degree $-1$, while the second factor in $\wedge A^*\otimes TM$ has degree $0$. 
\end{proposition}
\begin{proof}
Any derivation $D_m\in \mf{der}(\wedge A^*)_m$ is determined by its restriction to the degree $0$ and degree $1$ components of 
$\Gamma(\wedge A^*)$. There is a bundle map $\mf{der}(\wedge A^*)\otimes T^*M\to \wedge A^*$, taking $D_m\otimes (\d f)_m$ to 
$D_m(f)$ for $f\in C^\infty(M)$. This is well-defined, since $D_m(f)$ vanishes if $f$ is constant, by the derivation property. We may also regard this as a map 
\begin{equation}\label{eq:map2}\mf{der}(\wedge A^*)\to \wedge A^*\otimes TM.
\end{equation}
By construction, the kernel of  \eqref{eq:map2} at $m\in M$ is the subspace of 
derivations $D_m$ such that $D_m(f)=0$ for all  $f\in C^\infty(M)=\Gamma(\wedge^0 A^*)$. 
But this subspace is exactly $\wedge A^*_m\otimes A_m=\mf{der}(\wedge A^*_m)\subseteq \mf{der}(\wedge A^*)_m$, 
 where the factor $A_m$ corresponds to `contractions'. This defines an injective bundle morphism 
$\wedge A^*\otimes A\to \mf{der}(\wedge A^*)$ 
whose image is the kernel of \eqref{eq:map2}. For surjectivity of the $\wedge A^*$-module morphism \eqref{eq:map2}, it is enough to show surjectivity of the map 
$\mf{der}^0(\wedge A^*)\to TM$.  But any choice of a vector bundle connection on $A$ defines a splitting of this map. 
\end{proof}
\begin{remark}
We see in particular that $\mf{der}^i(\wedge A^*)$ vanishes for $i<-1$, and for $i=-1$ coincides with $A$, acting by contractions.  
In degree $i=0$ we obtain 
 the \emph{Atiyah algebroid} $\mf{aut}(A)$ of infinitesimal vector bundle automorphisms of $A$ (or equivalently of $A^*$), and the sequence \eqref{eq:exact2} becomes the usual exact sequence 
$0\to A^*\otimes A\to \mf{aut}(A)\to TM\to 0$ for the Atiyah algebroid.
\end{remark}

 Thinking of $\mf{der}(\wedge A^*)$ as a generalization
of the tangent bundle (to which it reduces if $\on{rank}(A)=0$), the
corresponding `cotangent bundle' is the graded $\wedge
A^*$-module
\begin{equation}\label{eq:om1} \Om^1_{\wedge A^*}=\Hom_{\wedge A^*}(\mf{der}({\wedge A^*}),{\wedge A^*})\end{equation}
of \emph{K\"ahler differentials}. Dual to \eqref{eq:exact2}, we obtain an exact sequence of graded $\wedge A^*$-modules
\[ 0\to \wedge A^*\otimes T^*M\to \Om^1_{\wedge A^*}\to \wedge A^*\otimes A^*\to 0.\]
Here the second factor in $\wedge A^*\otimes T^*M$ has degree $0$ while the second factor in 
$\wedge A^*\otimes A^*$ has degree $1$.
More generally, we define a \emph{module 
of K\"ahler q-forms} $\Om^q_{\wedge A^*}$ to be the  
$q$-th exterior power (taken over $\wedge A^*$). That is, $\Om^q_{\wedge A^*}$ consists of 
graded  bundle maps 
\begin{equation}\label{eq:skew} \mf{der}({\wedge A^*})\times\cdots \times \mf{der}({\wedge A^*})\to {\wedge A^*}
\end{equation}
(with $q$ factors) that are ${\wedge A^*}$-linear in each entry and skew-symmetric in the graded sense. For $q=0$ we put $\Om^0_{\wedge A^*}={\wedge A^*}$. 
Each $\Om^q_{\wedge A^*}$ is a graded $\wedge A^*$-module, with summands
\[ W^{p,q}(A):=(\Om^q_{\wedge A^*})^p\]
the $q$-linear maps \eqref{eq:skew} raising the total degree by $p$. The `wedge product' 
$\wedge \Om^q_{\wedge A^*}\otimes \Om^{q'}_{\wedge A^*}\to \Om^{q+q'}_{\wedge A^*}$ is 
compatible with these gradings, thus $W(A)=\Om_{\wedge A^*}$ is a bundle of bigraded algebras. 
We will denote by 
\[ \sW^{\bullet,\bullet}(A):=\Gamma(W^{\bullet,\bullet}(A))\]
the bigraded algebra of sections. From its interpretation as `differential forms', it is clear that this algebra has an exterior differential:
\begin{proposition}
The algebra $\sW^{\bullet,\bullet}(A)$ has a unique derivation $\d_K$ of bidegree $(0,1)$ such that $\d_K\circ \d_K=0$ and such that 
for all $\phi\in \Gamma(\wedge A^*)$ and all $D\in \Gamma(\mf{der}(\wedge A^*))$, 
\[ (\d_K \phi)(D)=D(\phi).\]
\end{proposition}
\begin{definition}
The bigraded algebra $\sW(A)$ with the differential $\d_K$ will be called the \emph{Koszul algebra} of the vector bundle $A\to M$.
\end{definition}
We list some properties and special cases of this construction. 
\begin{enumerate}
\item[a)] Suppose $M=\pt$, so that $A=V$ is a vector space. Then 
$\mf{der}(\wedge V^*)=\wedge V^*\otimes V$, 
where elements of $V=\mf{der}^{-1}(\wedge V^*)$ acts as contractions. Dualizing, 
$\Om^1_{\wedge V^*}=V^*\otimes \wedge V^*$ where 
the elements of the first factor $V^*$ have bidegree $(1,1)$, and more generally
$\Om^q_{\wedge V^*}=S^qV^*\otimes \wedge V^*$
where the elements of $S^q V^*$ have bidegree $(q,q)$. It follows that 
\[ W^{p,q}(V)=S^q V^*\otimes\wedge^{p-q} V^*.\]
The differential $\d_K$ takes generators of $\wedge ^1V^*$ to the corresponding generators of $S^1V^*$; it hence coincides with the standard Koszul differential. 
\item[b)] At the other extreme, if $A=M\times\{0\}$ is the zero vector bundle over $M$, then $\mf{der}(\wedge A^*)=TM$ 
is the tangent bundle, and $\Om^q_{\wedge A^*}=\wedge^q T^*M$. Hence $W^{p,q}(A)$ is zero for $p>0$, while 
$W^{0,q}(A)=\wedge^q T^*M$. 
\item[c)] For a direct product of vector bundles $A_1\to M_1$ and $A_2\to M_2$, one has 
\[ W(A_1\times A_2)=W(A_1)\boxtimes W(A_2)\]
(exterior tensor product of graded algebra bundles) with the sum of the differentials on the two factors. As a special case, if 
$A=M\times V$ is a trivial vector bundle, then 
\[
 W^{p,q}(M\times V)=\bigoplus_i \wedge^i T^*M\otimes S^{q-i} V^*\otimes \wedge^{p-q+i} V^*.
\]
For a general vector bundle $A$, since $W(A)|_U=W(A|_U)$ for open subsets $U\subseteq M$, this gives a description of $W(A)$ in terms of local trivializations. 
%
\item[d)] For any vector bundle $A\to M$, one has $W^{p,0}(A)=\wedge^p A^*$ while $W^{0,q}(A)=\wedge^q T^*M$. 
The space $\mathsf{W}^{p,q}(A)=\Gamma(W^{p,q}(A))$ is spanned by elements of the form 
\begin{equation} \label{eq:spanning} \psi_0\ \d_K\psi_1\cdots \d_K\psi_q\end{equation}
with sections $\psi_i\in \Gamma(\wedge^{p_i} A^*)$ satisfying $p_0+\ldots+p_q=p$. (This follows, e.g., by considering local trivializations as above.)
\item[e)]  Any morphism $A'\to A$ of vector bundles over $M$ induces a morphism of bigraded algebra bundles 
$W(A)\to W(A')$ compatible with the $\wedge T^*M$ module structure. The map on sections 
$\sW(A)\to \sW(A')$ is a cochain maps with respect to $\d_K$.
\item[g)] Let $f\colon N\to M$ be a smooth map. For any vector bundle $A\to M$, the algebra bundles $W(f^*A)$ and 
$f^*W(A)$ are related by `change of coefficients': 
\[ W(f^*A)=\wedge T^*N\otimes_{f^*\wedge T^*M} f^*W(A).\]
Thus, on the level of sections we have an inclusion $\Om(N)\otimes_{\Om(M)}\sW(A)\hra \sW(f^*A)$ 
with dense image. More generally, for any morphism of vector bundles $A_1\to A_2$ with underlying map 
$f\colon M_1\to M_2$ we obtain a morphism $f^*\colon \sW(A_2)\to \sW(A_1)$.
%
%
\end{enumerate}
The morphisms
\begin{equation}\label{eq:inclproj}
 i\colon \Om(M)\to \mathsf{W}(A),\ \ \ \pi\colon \mathsf{W}(A)\to \Om(M),
\end{equation}
induced by the projection $A\to M$ and the inclusion $M\to A$,  respectively, may be regarded as the inclusion and projection onto 
the subcomplex $W^{0,\bullet}(A)\cong \wedge^\bullet T^*M$.  
\begin{proposition}\label{prop:koszul}
The inclusion and projection \eqref{eq:inclproj}
are homotopy inverses with
respect to $\d_K$. In particular, the cohomology of $(\on{Tot}^\bullet\mathsf{W}(A),\d_K)$ is
canonically isomorphic to 
the de Rham cohomology of $M$.
\end{proposition}
\begin{proof}
View $A\times\R$ as the direct product of $A$ with the zero vector bundle $\R\times\{0\}$ over $\R$; 
thus $W(A\times \R)=W(A)\boxtimes \wedge T^*\R$. The space $\sW(A\times \R)=\Gamma(W(A\times\R))$ 
may be regarded as differential forms on $\R$ with values in $\sW(A)$. For 
all $s\in\R$ we have morphisms of bigraded algebras $\on{ev}_s\colon \sW(A\times\R)\to \sW(A)$ induced by the 
bundle map $A\to A\times\R,\ v\mapsto (v,s)$. Integration over the unit interval $[0,1]\subset\R$ defines a map
\[ J\colon \sW^{\bullet,\bullet}(A\times\R)\to \sW^{\bullet,\bullet-1}(A)\]
with the homotopy property (Stokes' theorem) 
\[ J\circ \d_K+\d_K\circ J=\on{ev}_1-\on{ev}_0.\]
The bundle map $A\times \R\to A,\ \ (v,t)\mapsto tv$
defines a morphism of bigraded algebras $F\colon \sW(A)\to \sW(A\times\R)$, with 
\[ \on{ev}_1\circ F=\on{id}_{\sW(A)},\ \ \ \ \on{ev}_0\circ F=i\circ \pi.\]
Since $F$ and the maps $\on{ev}_s$
commute with the differential $\d_K$,
it follows that the composition $J\circ F\colon \sW^{\bullet,\bullet}(A)\to \sW^{\bullet,\bullet-1}(A)$ is a homotopy operator between these two maps:
\[ J\circ F\circ \d_K+\d_K\circ J\circ F=\on{id}_{\sW(A)}-i\circ \pi.\]
(For a more detailed discussion, see e.g., \cite[Section 6.3]{me:clifbook}.) 
\end{proof}

\subsection{Derivations of $\sW(A)$}\label{subsec:deri}
In addition to the `exterior differential' $\d_K$, the algebra $\sW(A)$ has `Lie derivatives' $l(D)$ and `contractions' 
$j(D)$ defined by 
derivations $D\in \Gamma(\mf{der}^i(\wedge A^*))$. 
Here $j(D)$ is the derivation of bidegree $(i,-1)$ given on $\phi\in \Om^1_{\wedge A^*}$ 
(cf. \eqref{eq:om1}) by $j(D)\phi=\phi(D)$, while 
$l(D)$ is the derivation of bidegree $(i,0)$, extending 
$D$ on $\Gamma(\wedge^\bullet A^*)=\sW^{\bullet,0}(A)$ and commuting with $\d_K$ in the graded sense.
We have the Cartan commutation relations
\[ \begin{split}
[l(D_1),l(D_2)]&=l([D_1,D_2]),\\ 
[l(D_1),j(D_2)]&=j([D_1,D_2]),\\ 
[j(D_1),j(D_2)]&=0,\\
[l(D),\d_K]&=0,\\ 
[j(D),\d_K]&=l(D),\\ 
[\d_K,\d_K]&=0,\end{split}\]
for $D,D_1,D_2\in \Gamma(\mf{der}\,^\bullet(\wedge A^*))$.  The constructions are natural with respect to morphisms $A_1\to A_2$ 
of vector bundles: If  the map $f^*\colon \Gamma(\wedge A_2^*) \to
 \Gamma(\wedge A_1^*)$ satisfies $f^*\circ D_2=D_1\circ f^*$, then the 
 map $f^*\colon \mathsf{W}(A_2)\to
 \mathsf{W}(A_1)$ satisfies $f^*\circ j(D_2)=j(D_1)\circ f^*$ and
 $f^*\circ l(D_2)=l(D_1)\circ f^*$.

In particular, the derivations $\i(X)\in \Gamma(\mf{der}\,^{-1}(\wedge A^*))$ given by contraction with 
$X\in \Gamma(A)$ give rise to derivations
\[ \i_S(X):=j(\i(X)),\ \ \ \i_K(X):=l(\i(X))\]
of $\mathsf{W}(A)$, of bidegrees $(-1,-1)$ and $(-1,0)$ respectively.  In the special case $A=V$, so that $W(V)=S
V^*\otimes\wedge V^*$ is the standard Koszul algebra, 
$\i_K(X)$ is the contraction operator acting on the second factor while $\i_S(X)$ is the contraction operator on the first 
factor.  
We have $[\i_S(X),\d_K]=\i_K(X)$ and $[\i_K(X),\d_K]=0$. Note also that for $f\in C^\infty(M)$, 
$\i_S(fX)=f\i_S(X)$ but 
\begin{equation}\label{eq:nottensorial}
\i_K(fX)=f\i_K(X)-\d f\circ \i_S(X)
\end{equation} 
where $\d f\in \Om^1(M)=\sW^{0,1}(A)$ acts by multiplication. 

\begin{remark}
There is an alternative geometric model for the Koszul algebra of a vector bundle $A\to M$, as follows. For $p\ge 0$ let 
$A^{(p)}=A\times_M\cdots \times_M A$ be the $p$-fold fiber product over $M$, with the convention $A^{(0)}=M$.
Thinking of $A$ as a groupoid and of $A^{(p)}$ as $B_pA$, we have a cup product on $C^\infty(A^{(\bullet)})$. We let $C_{sk}^\infty(A^{(\bullet)})\subset C^\infty(A^{(\bullet)})$ denote the subspace of skew-symmetric functions, endowed with the multiplication given by the skew-symmetrization of the cup product.
There is an injective morphism of graded algebras 
\[ \Gamma(\wedge^\bullet A^*)\to C_{sk}^\infty(A^{(\bullet)}),\]
taking a section of the exterior power $\wedge^p A^*$  to the corresponding multi-linear, skew-symmetric 
function on $A^{(p)}$. 

In a similar fashion, let $\Om_{sk}^\bullet(A^{(\bullet)})\subset \Om^\bullet(A^{(\bullet)})$ denote the subspace of forms which are skew-symmetric (for the action of the symmetric group $\mf{S}_p$), endowed with the skew-symmetrized cup product. There is an injective morphism of bigraded algebras 
\[ \sW^{\bullet,\bullet}(A)\to \Om_{sk}^\bullet(A^{(\bullet)}),\]
taking a section of $W^{p,q}(A)$ to a q-form on $A^{(p)}$ that is 
multi-linear (i.e., linear in each factor). This morphism intertwines the Koszul differential $\d_K$ with the de Rham 
differential. In particular, $\sW^{1,q}(A)$ is realized as the space of linear $q$-forms on $A$. This space plays a role in the work of 
Bursztyn-Cabrera-Ortiz \cite{bur:mul,bur:lin} on multiplicative 2-forms. 
\end{remark}

\subsection{The Weil algebroid of a Lie algebroid}
Suppose now that $A\to M$ is a Lie algebroid. The Chevalley-Eilenberg differential 
$\d_{CE}$ on sections of $\wedge A^*$ lifts to a differential $l(\d_{CE})$ on sections of $W(A)$. 
Like all operators of the form $l(D)$, it commutes with $\d_K$ in the graded sense. 
To simplify notation, we will write $l(\d_{CE})=\d_{CE}$. 
\begin{definition}
The bidifferential algebra $(\mathsf{W}(A),\d_K,\d_{\on{CE}})$ 
is called the \emph{Weil algebra} of the Lie algebroid $A$. The total differential 
$\d_W=\d_K+\d_{\on{CE}}$ is called the \emph{Weil differential}. 
\end{definition}
For any Lie algebroid morphism $A_1\to A_2$, the resulting map 
$f^*\colon \Gamma(\wedge A_2^*)\to \Gamma(\wedge A_1^*)$ intertwines the derivations $\d_{CE}$, hence 
$f^*\colon \sW(A_2)\to \sW(A_1)$ is a morphism of bidifferential algebras. 

Let $A\to M$ be a Lie algebroid, with Weil algebra $\sW(A)$. For a section $X\in \Gamma(A)$, we obtain a degree zero derivation 
$\L(X)=[\i(X),\d_{CE}]$ of $\Gamma(\wedge A^*)$; its extension to $\sW(A)$ will again be denoted by $\L(X)$.
We obtain yet another contraction operator $\i_{CE}(X):=j(\L(X))$, of bidegree $(0,-1)$. From the Cartan commutation relations, 
we see that 
\[ [\i_K(X),\d_{CE}]=\L(X)=[\i_{CE}(X),\d_K],\ \ \ [\i_{CE}(X_1),\i_K(X_2)]=\i_{S}([X_1,X_2]).\]
\subsection{Examples}
\begin{example}\label{ex:se0}
Consider first the case that $M=\pt$, so that $A=\g$ is a Lie algebra.  Choose dual bases  $e_i\in \g$ and $e^i\in \g^*$, and let 
$c_{ij}^k=\l e^k, [e_i,e_j]\r$ be the structure constants. The Chevalley-Eilenberg differential on $\wedge \g^*$ 
is given by the formula
$\d_{\on{CE}}=-\hh \sum_{ijk} c_{ij}^k \ e^i  e^j \i(e_k)$, with $\i(e_k)$ the contraction operator on $\wedge\g^*$. 
As we had seen, $W^{p,q}(\g)=S^q\g^*\otimes\wedge^{p-q}\g^*$, with $\d_K$ the standard Koszul differential. 
Letting $\ol{e}^i\in S^1\g^*$ denote the degree $(1,1)$ generators corresponding to the basis elements, 
we have $\d_K=\sum_i \ol{e}^i \iota(e_i)$. The operator $j(\d_{CE})$ on the Weil algebra becomes
\[ j(\d_{\on{CE}})=-\hh \sum_{ijk} c_{ij}^k \ e^i  e^j \i_S(e_k),\]
hence the differential $\d_{CE}:=l(\d_{CE})=[j(\d_{CE}),\d_K]$ on $W(\g)$ is 
\[ \d_{CE}=-\hh \sum_{ijk} c_{ij}^k \ e^i  e^j \iota(e_k)
+\sum_{ijk} c_{ij}^k \ e^i  \ol{e^j} \i_S(e_k).\]
One recognizes $(W(\g),\d_K,\d_{CE})$ as the standard Weil algebra. Here 
$\i_K(e_k)=\iota(e_k)$ is the usual contraction on the $\wedge\g^*$ factor, $\i_S(e_k)$ is the usual contraction on the 
$S\g^*$ factor, and $\i_{CE}(e_k)=\sum c_{ik}^j e^i \i_S(e_j)$.
\end{example}

\begin{example}\label{ex:se1}(Lie algebroid structures on trivial vector bundles)
Let $A\to M$ be a Lie algebroid, with a trivialization $A=M\times V$ as a vector bundle. Thus 
$\sW(A)=\Om(M)\otimes SV^*\otimes \wedge V^*$. 
Choose dual bases $e_i\in V$ and $e^i\in V^*$. Viewing the $e_i$ as constant sections of $A$, put 
$c_{ij}^k=\l e^k, [e_i,e_j]\r\in C^\infty(M)$. By a calculation similar to that of example \ref{ex:se0}, we obtain
the following formula for the Chevalley-Eilenberg differential on $\sW(A)$, 
\[ \d_{\on{CE}}=\sum_i  e^i\ \L_M(\a(e_i))-\sum_i \ol{e^i} \i_M(\a(e_i))
-\hh \sum_{ijk} c_{ij}^k \ e^i  e^j \i(e_k)
+
\sum_{ijk} c_{ij}^k \ e^i  \ol{e^j} \i_S(e_k).\]
Here $\i_M(\a(e_i))$ and $\L_M(\a(e_i))$ are contraction and Lie derivative with respect to the vector field
$\a(e_i)$, acting on the $\Om(M)$ factor, $\i(e_k)$ is a contraction on the $\wedge V^*$ factor, 
$\i_S(e_k)$ is contraction on the $S V^*$ factor, 
and the $e^i, \ol{e^i}$ are the generators of $\wedge V^*$ and $S V^*$, acting by multiplication.

The special case that the $c_{ij}^k$ are constant corresponds to an action Lie algebroid for an action of the Lie algebra 
$V=\g$ on $M$. Here $(\sC(A),\d_{CE})$ is the Chevalley-Eilenberg complex of $\g$ with coefficients in 
$C^\infty(M)$, and $(\mathsf{W}(A),\d_W)$ is isomorphic to $W(\g)\otimes \Om(M)$
with differential $\d_{W\g}\otimes 1+1\otimes \d_M$, using the isomorphism given by a Kalkman twist  by the operator $\exp(\sum_i e^i\otimes \i_M(e_i))$. See Guillemin-Sternberg \cite{gu:sy} and Abad-Crainic \cite{aba:wei}. 
\end{example}

\begin{example}(Tangent bundle)\label{ex:gorms}
If $A=TM$, the Chevalley-Eilenberg complex $\Gamma(\wedge A^*)=\Om(M)$ is the usual de Rham complex. Thus, 
$\mathsf{W}^{p,q}(TM)$ comes with two kinds of de Rham differentials, 
$\d'=\d_{\on{CE}}$ and $\d''=\d_K$. As a bigraded algebra, 
the Weil algebra $\mathsf{W}(TM)$ is generated by functions $f\in C^\infty(M)$, 
$(1,0)$-forms
$\d'f$, $(0,1)$-forms $\d'' f$, and $(1,1)$-forms
$\d'\d'' f$. The bidifferential algebra 
\begin{equation}\label{eq:gorms}
\Om_{[2]}(M):=\mathsf{W}(TM)
\end{equation}
with differentials $\d',\d''$  
was introduced by Kochan-Severa \cite{koc:dif} under the name of \emph{differential gorms}; it was subsequently studied by 
Vinogradov-Vitagliano \cite{vin:ite} under the name of \emph{iterated differential forms}. (Obviously, there are generalizations to 
$n$-differential algebras  $\Om_{[n]}(M)$.) Many of the standard constructions for differential forms generalize with minor changes. 
In particular, iterated differential forms can be pulled back under smooth maps, and given a smooth homotopy $F\colon [0,1]\times M\to N,\ (t,x)\mapsto F_t(x)$ one obtains two homotopy operators $h',\,h''\colon \Om_{[2]}(N)\to \Om_{[2]}(M)$, of 
bidegrees $(-1,0)$ and $(0,-1)$, such that $[\d',h']=[\d'',h'']=F_1^*-F_0^*$ while 
$[\d',h'']=[\d'',h']=0$. The homotopy operators are obtained as pullbacks under the map $F$, followed by integration over $[0,1]$ with respect to $\d' t$, respectively $\d'' t$.
\end{example}

\begin{example}(Foliations)
Suppose  $\ca{F}$ is a foliation of $M$, defining a Lie algebroid $A=T_\F M$. The inclusion $T_\F M\to TM$ 
defines a surjective map from \eqref{eq:gorms} onto the Weil algebra $\sW(T_\F M)$. 
One can think of elements of $\sW(T_\F M)$ as differential gorms 
in the direction of the foliation and differential forms in the transverse direction.
\end{example}

Similar to the well-known result for the Weil algebra $W(\g)$, we have: 
\begin{proposition}
For any Lie algebroid $A\to M$, there is a canonical homotopy equivalence between  
$(\on{Tot}^\bullet\mathsf{W}(A),\d_W)$ and the de Rham algebra $(\Omega^\bullet(M),\d_M)$. 
\end{proposition}
\begin{proof}
The proof is a generalization of the `Kalkman trick'. The 
derivation $\mf{u}=j(\d_{\on{CE}})$ has bidegree $(1,-1)$, and satisfies 
\[ [\mf{u},\d_K]=\d_{\on{CE}},\ \ [\mf{u},\d_{\on{CE}}]=0.\]
Since $\mf{u}$ has total degree $0$ and is nilpotent, its exponential $U=\mf{u}$ 
is a well-defined algebra automorphism of $\mathsf{W}(A)$, preserving the total degree, and with 
\[ U\circ \d_K\circ U^{-1}=\d_K+\d_{\on{CE}}=\d_W.\]
By Proposition \ref{prop:koszul}, the inclusion $\Om^\bullet(M)\hra
\on{Tot}^\bullet\mathsf{W}(A)$ is a homotopy equivalence with respect to $\d_K$; hence
its composition with $U$  is a homotopy
equivalence with respect to $\d_W$.
\end{proof}

\section{The Van Est map $\Om(BG)\to \sW(A)$}
We will now continue the discussion from Section \ref{sec:VEd} to define a Van Est map for the Weil algebras.
\subsection{The Van Est triple complex}
The simplicial Lie algebroid $T_\F E_\bullet G\to E_\bullet G$ gives rise to a tridifferential algebra $\sW(T_\F EG)$, with summands
$\sW^{r,p,q}(T_\F EG)=\sW^{p,q}(T_\F E_rG)$, 
and with commuting differentials 
\[\delta,\ \ \d'=(-1)^r\d_{CE},\ \ \d''=(-1)^r\d_K\] 
of tridegrees $(1,0,0),\ (0,1,0)$, and $(0,0,1)$. The product is a cup product, as before:
\[ \alpha\cup\alpha'=(-1)^{r'(p+q)}\pr^*\alpha\ (\pr')^*\alpha'\]
for $\alpha\in \sW^{p,q}(T_\F E_rG)$ and $\alpha'\in, \sW^{p',q'}(T_\F E_{r'}G)$, 
where the right hand side uses the multiplication in $\sW^{\bullet,\bullet}(T_\F E_{r+r'}G)$.
We have a diagram, for all $r$, 
\[ \xymatrix{ {\sW^{\bullet,\bullet}(T_\F E_rG)} & {\Om^\bullet(B_rG)}\ar[l]_{\ \ \ \ \ \kappa_r^*} \\
{\sW^{\bullet,\bullet}(A_r)} \ar[u]^{\pi_r^*} &}.
\]
Both $\kappa_\bullet^*$ and $\pi_\bullet^*$ are morphisms of 
tridifferential algebras, where $\Om^\bullet(B_\bullet G)$ is regarded as a triple complex
concentrated in tridegrees $(\bullet,0,\bullet)$.  We also have the maps 
\[ \iota_r^*\colon \sW^{\bullet,\bullet}(T_\F E_rG)\to \sW^{\bullet,\bullet}(A_r)\]
induced by the inclusion $\iota_r\colon A_r\to T_\F E_rG$. Then $\iota_\bullet^*$ is a left inverse to $\pi_\bullet^*$ intertwining the simplicial differential $\delta$ as well as the Koszul differential $\d''$, but usually not the differential $\d'$.

\subsection{The Van Est map for the Bott-Shulman-Stasheff complex}
Since the maps $h_{r,i}\colon E_rG\to E_{r+1}G$ lift to vector bundle morphisms $T_\F E_rG\to T_\F E_{r+1}G$,  we have a well-defined homotopy operator  $h=\sum_i (-1)^{i+1} (h_{r-1,i})^*\colon \sW(A_r)\to \sW(A_{r-1})$ 
with respect to the simplicial differential $\delta$. On the dense subspace 
\[ \Om(E_rG)\otimes_{\Om(M)} \sW(A)\subseteq \sW(T_\F E_rG),\]
it is the natural extension of the homotopy operator on $\Om(E_\bullet G)$. (This is well-defined, since the latter is a
$\Om(M)$-module morphism, cf.~ Part \eqref{it:mod} of Proposition \ref{prop:homproperties}.) Note that $h$ commutes with $\d''$, but usually not with 
$\d'$. Let $\on{Tot}_{12}^{\bullet,\bullet}\sW(T_\F EG)$ be the bidifferential algebra with summands 
$\on{Tot}_{12}^{n,q}\sW(T_\F EG)=\bigoplus_{r+p=n}\sW^{p,q}(T_\F E_rG)$, and with the differentials 
$\delta+\d'$ and $\d''$. We denote by $\on{Tot}^\bullet \sW(T_\F EG)$ the total complex obtained by summing over all three gradings. 
\begin{proposition}
The composition 
\[\iota_0^*\circ (1+\d'\circ h)^{-1}\colon \on{Tot}_{12}^{\bullet,\bullet}\sW(T_\F EG)\to \sW^{\bullet,\bullet}(A)\] 
is a
morphism of bidifferential spaces. In fact, it is a homotopy equivalence with respect to $\delta+\d'$, with homotopy inverse 
$\pi_0^*$. It restricts to an algebra morphism on the normalized subcomplex $\on{Tot}_{12}^{\bullet,\bullet}\wt{\sW}(T_\F EG)$.
\end{proposition}
\begin{proof}
The first part is a direct consequence of the Perturbation Lemma \ref{lem:A}, applied to $\on{Tot}_{12}^{\bullet,q}\sW(T_\F EG)$ for 
fixed $q$. We obtain a similar statement for the total complex $\on{Tot}^\bullet \sW(T_\F EG)$ (with the differential $\delta+\d$ where $\d=\d'+\d''$), 
for the composition $\iota_0^*\circ (1+\d\circ h)^{-1}$. By Lemma \ref{lem:B} (cf.~ the proof of 
Theorem \ref{th:multi}), the map $\iota_0^*\circ (1+\d\circ h)^{-1}$ is an algebra morphism on normalized cochains. 
But this map coincides with $\iota_0^*\circ (1+\d'\circ h)^{-1}$, because 
\[  (1+\d\circ h)^{-1}=(1+\d'\circ h-h\circ \d'')^{-1}=(1+\d'\circ h)^{-1}+\sum_{n=1}^\infty (-h\circ \d'')^n\]
(using that $h$ and $\d''$ commute), and $\iota_0^*\circ h=0$. 
\end{proof}
\begin{definition}
The composition 
\[ \on{VE}\colon \iota_0^*\circ (1+\d'\circ h)^{-1}\circ \kappa^*\colon \Om^\bullet(B_\bullet G)\to \sW^{\bullet,\bullet}(A).\]
is the Van Est map for the Bott-Shulman-Stasheff  double complex.
\end{definition}
By construction, the map $\on{VE}$ is a morphism of bidifferential spaces, and it restricts to an algebra morphism on the normalized 
cochains. It is an $\Om(M)$-module morphism, since each of the maps $\iota_0^*,\ \kappa^*$, and $1+\d'\circ h$ is an $\Om(M)$-module morphism.

For local Lie groupoids $G$, one similarly obtains a Van Est map on the complex of germs, 
\[ \on{VE}_M\colon \Om^\bullet(B_\bullet
G)_M\to \sW^{\bullet,\bullet}(A).\]
The latter is surjective, and as we shall see in the next section, admits a right inverse  which is a morphism of bidifferential spaces. 
The Van Est map 
for a global Lie groupoid $G$ factors through the localized Van Est map $\on{VE}_M$. 

\section{Van Est theorems}
The Van Est map can be viewed as a differentiation procedure from Lie groupoid cochains to Lie algebroid cochains. In 
some situations, it is possible to obtain an integration procedure in the opposite direction.  In our approach, 
the Van Est map was constructed using a homotopy operator with respect to $\delta$; to obtain a cochain map in the other direction one wants a homotopy operator with respect to the differential $\d$.

Note that the principal $G$-bundles $\kappa_p\colon E_pG\to B_pG$ are trivial: For any fixed $i\le p$, 
the submanifold of elements $(a_0,\ldots,a_p)\in E_pG$ with $a_i\in M$  defines 
a section. Taking $i=0$, the corresponding right inverse to $\kappa_p$ is the map 
\[ j_p\colon B_pG\to E_pG,\ \ \  (g_1,\ldots,g_p)\mapsto (\tz(g_1),g_1^{-1},\ldots,(g_1\cdots g_p)^{-1}).\]
As before, we regard $\Om^\bullet(B_\bullet G)$ as a bidifferential algebra concentrated in bidegrees $(\bullet,0,\bullet)$. The 
morphism of bigraded spaces
\[ j_\bullet^*\colon \sW^{\bullet,\bullet}(T_\F E_\bullet G)\to \Om^\bullet(B_\bullet G)\]
(given by the obvious pullback map in tridegree $(\bullet,0,\bullet)$, and
equal to zero  in all other tridegrees)  
is a left inverse to $\kappa_\bullet^*$. It is a cochain map with respect to $\d',\d''$ (in particular, $j_\bullet^*\circ \d'=0$), 
but since $j_\bullet$ is \emph{not} a simplicial map it is neither a cochain map with respect to $\delta$, 
nor an algebra morphism.

Consider the very special case that the $\tz$-fibers of $G$ are contractible, in the sense that there is a smooth deformation retraction 
$\lambda_t\colon G\to G$, depending smoothly on $(t,g)\in  [0,1]\times G$, and 
such that 
\begin{equation}\label{eq:retractionproperties}
\lambda_t|_M=\on{id}_M,\ \ \lambda_0=\on{id}_G,\ \ \lambda_1=\iota\circ \tz,\ \ \tz\circ \lambda_t=\tz
\end{equation}
for all $t\in [0,1],\ g\in G$.  
One then obtains deformation retractions $\lambda_{p,t}\colon E_pG\to E_pG$ 
with 
\[ \lambda_{p,t}|_{B_pG}=\on{id}_{B_pG},\ \ \lambda_{p,0}=\on{id}_{E_pG},\ \ \lambda_{p,1}=j_p\circ \kappa_p,\ \ \kappa_p\circ \lambda_{p,t}=\kappa_p,\]
by the formula
\[ \lambda_{p,t}(a_0,\ldots,a_p)=\big(\lambda_t(a_0),\,a_1 a_0^{-1} \lambda_t(a_0),\ldots,\,a_p a_0^{-1} \lambda_t(a_0)\big).\]
 In turn, 
these define homotopy operators (cf.~ Example \ref{ex:gorms})
\[ k\colon \sW^{p.q}(T_\F E_rG)\to \sW^{p-1,q}(T_\F E_rG)\]
(i.e.,  $k\d'+\d' k=\on{id}-\kappa_\bullet^*\,j_\bullet^*$), with $k\d''+\d'' k=0$. 
%
%

For a general Lie groupoid $G$, or even a local Lie groupoid, one can always choose a \emph{germ} of a deformation 
retraction $\lambda$ along the $\tz$-fibers. The properties \eqref{eq:retractionproperties} are 
to be understood as equalities of germs along $M$ (or along $[0,1]\times M$). 
The germ determines a homotopy operator $k_r\colon \sW^{p,q}(T_\F E_rG)_M\to \sW^{p-1,q}(T_\F E_rG)_M$ for the complex of germs.
We obtain:
\begin{proposition}\label{prop:localInv}
For any local Lie groupoid $G\rra M$
the map 
$\on{VE}_M\colon \Om^q(B_\bullet G)_M\to \sW^{\bullet,q}(A)$
is a homotopy equivalence, for all fixed $q$. Given a germ of a retraction of $G$ onto $M$ along $\tz$-fibers, the 
corresponding operator $k$  defines a homotopy inverse:
\[ j_\bullet^* \circ (1+\delta k)^{-1}\circ \pi^*_0\colon \sW^{\bullet,\bullet}(A)\to \Om^\bullet(B_\bullet G)_M.\]
Similar assertions hold for the Van Est map $\on{VE}$ of global Lie groupoids with contractible $\tz$-fibers. 
\end{proposition}
\begin{proof}
Reversing the roles of $\d$ and $\delta$ in the Perturbation Lemma \ref{lem:A}, we see that 
\[ j_\bullet^* \circ (1+\delta k)^{-1}\colon \on{Tot}^\bullet \sW(T_\F EG)_M\to \Om^\bullet(B_\bullet G)_M\]
is a cochain map, and is a homotopy inverse to $(1+k\delta)^{-1}\kappa^*=\kappa^*$.
Here we used that $k\delta$ vanishes on the range of $\kappa^*$, for degree reasons. 
On the other hand, by Proposition \ref{prop:hmk}, the map $\iota_0^*\circ  (1+\d\circ h)^{-1}$ is homotopy inverse to 
$\pi^*_0$. 
\end{proof}
\begin{remark}
Once again, we can write this `reverse Van Est map' as a zig-zag: In bidegree $(p,q)$, it reads as 
\[ (-1)^p j_p^* \circ (\delta k)^p\circ \pi_0^*\colon 
\sW^{p,q}(A)\to \Om^q(B_p G)_M.\]
\end{remark}
The following result is due to Weinstein-Xu \cite{wei:ext} in the case $q=0$, and to Bursztyn-Cabrera \cite{bur:mul}
in the general case.
\begin{proposition}
Let $G\rra M$ be a local Lie groupoid. 
In bidegrees $(p,q)$ with $p=0,1$, the map $\on{VE}_M\colon \Om^q(B_p G)_M\to \sW^{p,q}(A)$ 
restricts to an isomorphism on $\delta$-cocycles. 
Similar assertions hold for  global Lie groupoids with $1$-connected $\tz$-fibers.
\end{proposition}
\begin{proof}
On $\Om^q(B_0 G)_M=\sW^{0,q}(A)_M=\Om^q(M)$, the map $\on{VE}_M$ is just the identity map. 
The space $\ker(\delta)\subseteq \sW^{0,q}(A)_M$ 
consists of (locally) $G$-invariant q-forms, while $\ker(\d')$
consists of $q$-forms that are $A$-invariant.
But these two spaces coincide. It follows that  $\on{VE}_M$ restricts to an isomorphism on $\delta$-cocycles in bidegree 
$(0,q)$, as well as on  $\delta$-coboundaries in bidegree 
$(1,q)$. Since $\on{VE}_M$ induces an isomorphism in cohomology for the differentials $\delta,\d'$, it must then also restrict to an isomorphism 
on 1-cocycles. For global Lie groupoids $G\rra M$, consider the quotient map $\Om^q(B_p G)\to 
\Om^q(B_p G)_M$.
A $\delta$-cocycle in $\Om^q(B_0 G)$ is a (globally) $G$-invariant form; if $G$ is $0$-connected 
this is the same as a locally $G$-invariant form, i.e. a cocycle in $\Om^q(B_0 G)_M$. 
A $\delta$-cocycle in $\Om^q(B_1 G)$ is a multiplicative form on $G$.
Such a form is uniquely determined by its restriction to an arbitrarily small open 
neighborhood of $M$ in $G$, i.e., by its germ. Hence the map $\Om^q(B_1 G)\to \Om^q(B_1 G)_M$ is injective on $\delta$-cocycles. 
If the $\tz$-fibers are $1$-connected, then any germ (along $M$) of a multiplicative form extends uniquely 
to a global multiplicative form.  Hence the map is also surjective in that case. 
\end{proof}

\begin{remark}
The prescription in \cite{wei:ext}  is equivalent to the one given here:
Any cocycle $\alpha\in \sC^1(A)=\Gamma(A^*)$ defines a closed left-invariant foliated 1-form 
$\alpha^L\in \Om^1_\F(G)$, for the foliation given by the target map. If the $\tz$-fibers are simply connected, 
one obtains a well-defined function $f\in C^\infty(G)$, such that $f(g)$ is the integral of $\alpha^L$  from 
$\tz(g)$ to $g$, along any path in the $\tz$-fiber. This function $f$ is multiplicative.
\end{remark}

For a global Lie groupoid, one has Crainic's Van Est theorem:
\begin{theorem}[Crainic \cite{cra:dif}] 
Suppose $G\rra M$ is a Lie groupoid with $n$-connected $\tz$-fibers. 
Then the Van Est map $\on{VE}\colon C^\infty(B_\bullet G)\to \sC^\bullet(A)$ induces an isomorphism in cohomology 
in degrees $p\le n$. For $p=n+1$ the map in cohomology is injective, with image the classes $[\om]$ such that for all $x\in M$, the 
integral of $\om$ (regarded as a left-invariant foliated form) over any $n+1$-sphere in $\tz^{-1}(x)$ is zero. 
\end{theorem}
(A generalization to $\Om(BG)$ was obtained by Abad-Crainic in \cite{aba:wei}.)
Using the homological perturbation theory, one can construct the inverse in degrees $\le n$ on the level of cochains, given a homotopy 
operator. The assumption that the $\tz$-fibers are $n$-connected implies that the fibers of any principal $G$-bundle are 
$n$-connected. In particular, this applies to $\kappa_r\colon E_rG\to B_rG$. It follows that
$\sC^{\bullet}(T_\F E_\bullet G)$ has vanishing $\d$-cohomology in bidegree $(r,s)$ for all $s\le n$.
Let 
\[ \tau_{\le n}\sC^{\bullet}(T_\F E_\bullet G)\]
be the truncated foliated de Rham complex for $G$, given by $\sC^{s}(T_\F E_rG)$ in degree $s<n$ and by 
$\sC^{n}(T_\F E_rG)\cap \ker(\d^n)$ in degree $n$. The truncated complex has vanishing $\d$-cohomology 
in degrees $(r,s)$ with $s>0$.  Hence there exists a homotopy operator 
\[ k\colon \tau_{\le n}\sC^{s}(T_\F E_rG)\to \tau_{\le n}\sC^{s-1}(T_\F E_rG)\]
with $k\d+\d k=\on{id}-\kappa_r^*j_r^*$. By the Perturbation Lemma, the composition 
\[ j^*\circ (1+\delta k)^{-1}\colon \tau_{\le n}\sC^{s}(T_\F E_rG) \to  C^\infty(B_r G)\]
is a cochain map for the total differential. It gives the desired cochain map
\[ j^*\circ (1+\delta k)^{-1}\circ \pi^*\colon \tau_{\le n}\sC^p(A) \to  C^\infty(B_p G).\]

\section{Explicit formulas for the Van Est map}
Until now, we expressed the Van Est map in terms of the Van Est double complex. We will now derive more explicit 
formulas, thus confirming that this definition agrees with those of Weinstein-Xu \cite{we:sy} and Abad-Crainic \cite{aba:wei}.  We will directly consider $\Om^\bullet(B_\bullet G)$; the results for $C^\infty(B_\bullet G)$ will be special cases. 
\subsection{The Lie algebroid $T_\F G$}
Let $G$ be a Lie groupoid with Lie algebroid $A=\on{Lie}(G)$. 
Let $\ca{F}$ be the foliation of $E_0G=G$ defined by the 
submersion $\kappa_0=\tz$, and let $T_\F G$ be the corresponding Lie algebroid. 
Recall that any $X\in \Gamma(A)$ induces derivations 
\[ \i_S(X),\ \i_K(X),\ \i_{CE}(X),\ \L(X)\] 
on $\sW^{\bullet,\bullet}(A)$. 
The left-invariant vector field $X^L\in \Gamma(T_\F G)$ defines similar derivations of $\sW^{\bullet,\bullet}(T_\F G)$.
The inclusion $\iota\colon M\to G$ lifts to a morphism of vector bundles 
$ A\to T_\F G$, defining a pullback map 
$\iota^*\colon \sW^{\bullet,\bullet}(T_\F G)\to \sW^{\bullet,\bullet}(A)$, with 
\[ \iota^*\circ \d_K=\d_K\circ \iota^*,\ \ \ \iota^*\circ \i_S(X^L)=\i_S(X)\circ \iota^*,\ \ \ \iota^*\circ \i_K(X^L)=\i_K(X)\circ \iota^*.\]
On the other hand, since $A\to T_\F G$ is not a Lie algebroid morphism, the map $\iota^*$ does not intertwine 
$\d_{CE},\ \L(X),\ \iota_{CE}(X)$ (for $X\in\Gamma(A)$) with the corresponding derivations of $\sW(T_\F G)$, in general. 
Instead we have
\begin{lemma} For all $X\in\Gamma(A)$, 
\[ \iota^*\circ \i_{CE}(X^L-X^R)=\i_{CE}(X)\circ \iota^*,\ \ \ \iota^*\circ \L(X^L-X^R)=\L(X)\circ \iota^*.\]
\end{lemma}
To explain the the left hand side of these equations, note that any vector field $Y\in \mf{X}(G)$ in the normalizer of $\Gamma(T_\F G)$ (i.e., such that $[Y,\cdot]$ preserves $\Gamma(T_\F G)$) defines an infinitesimal automorphism of $T_\F G$, giving rise to a derivation $\L(Y)$ of $\Gamma(\wedge T_\F^* G)$, and  hence to derivations $\i_{CE}(Y)=j(\L(Y))$ and $\L(Y)=l(\L(Y))$ 
of $\sW^{\bullet,\bullet}(A)$. This applies to the vector fields $X^L$ as well as to the vector fields $X^R$, hence also 
to the vector field $Y=X^L-X^R$ (generating the adjoint action). The Lemma follows since 
$[X^L-X^R,\cdot]$ on $\Gamma(T_\F G)$ induces $[X,\cdot]$ on $\Gamma(A)$.
It will be convenient to introduce the operator 
\begin{equation}\label{eq:D}
 \D\colon \sW^{\bullet,\bullet}(T_\F G)\to \sW^{\bullet+1,\bullet}(A),\ \ \D=\d_{CE}\circ \iota^*-\iota^*\circ \d_{CE},\end{equation}
measuring the failure of $\iota^*$ to be a cochain map for $\d_{CE}$. 
\begin{lemma}\label{lem:contractors}
For all $X\in\Gamma(A)$,
\begin{equation}\label{eq:contractors}
\begin{split}
\i_K(X)\circ \D+\D\circ \i_K(X^L)&=\iota^*\circ \L(-X^R),\\ 
\i_S(X)\circ \D-\D\circ \i_S(X^L)&=\iota^*\circ \i_{CE}(-X^R).
\end{split}
\end{equation}
\end{lemma}
\begin{proof}
Using the above commutation relations we calculate
\[ \begin{split}\i_S(X)\circ \D&=\i_S(X)\circ (\d_{CE}\circ \iota^*-\iota^*\circ \d_{CE})\\
&=(\i_{CE}(X)+\d_{CE}\circ \i_S(X))\circ \iota^*-\iota^*\circ \i_S(X^L)\circ \d_{CE}\\
&=\iota^*\circ \i_{CE}(X^L-X^R)+\d_{CE}\circ \iota^*\circ \i_S(X^L)-\iota^*\circ \i_{CE}(X^L)-\iota^*\circ \d_{CE}\circ \i_S(X^L)\\
&=\iota^*\circ \i_{CE}(-X^R)+\D\circ \i_S(X^L).
\end{split}\]
which proves the second identity. The first follows by taking a commutator with $\d_K$.
\end{proof}
On elements $\phi\in \Om^q(G)=\sW^{0,\bullet}(T_\F G)$, 
these formulas become (for degree reasons)
\[ \begin{split}
\i_K(X)\D\phi&=\iota^*\circ \L(-X^R)\phi\in \Om^q(M),\\ 
\i_S(X) \D\phi&=\iota^*\circ \i(-X^R)\phi\in \Om^{q-1}(M),\end{split}\]
where $\i(-X^R)$ is the usual contraction operator on differential forms.

\subsection{A formula for the Van Est map}\label{sec:wexu}
The vector fields $-X^{i,R}\in\mf{X}(E_rG)$ are invariant under the principal $G$-action, hence they 
descend to vector fields $X^{i,\sharp}\in\mf{X}(B_rG)$. The $-X^{i,R}$ generate the $G$-action on 
$E_rG$ given by left multiplication on the $i$-th factor; similarly the $X^{i,\sharp}$ generate the 
following $G$-actions on $B_rG$,
\[ g.(g_1,\ldots,g_r)=(g_1,\ldots,g_{i-1},g_i g^{-1},g g_{i+1},g_{i+2},\ldots,g_r).\]
These define Lie derivatives and contractions on $\Om(B_rG)$, with 
\[ \kappa_r^*\circ \i(X^{i,\sharp})=\i_K(-X^{i,R})\circ \kappa_r^*,\ \ \ 
\kappa_r^*\circ \L(X^{i,\sharp})=\L(-X^{i,R})\circ \kappa_r^*.\]
For elements $\alpha\in \sW^{p,q}(A)$, $X_1,\ldots, X_p\in \Gamma(A)$ and all 
$n\le p$ we put
\[ \alpha(X_1,\ldots,X_n,\ol{X}_{n+1},\ldots,\ol{X}_{p}) 
=\i_S(X_p)\cdots \i_S(X_{n+1})\i_K(X_n)\cdots \i_K(X_1)\alpha\in \Om^{q-n}(M).\]
This expression is $C^\infty(M)$-linear in $X_1,\ldots,X_n$, but not in $X_{n+1},\ldots,X_p$, due to \eqref{eq:nottensorial}.
\begin{theorem}\label{th:explicit}
The Van Est map $\on{VE}\colon \Om^\bullet(B_\bullet G)\to \sW^{\bullet,\bullet}(A)$ is given by the following formula, 
for $\phi\in \Om^q(B_pG)$ and $X_1,\ldots,X_p\in\Gamma(A)$,
\[ \begin{split}
\lefteqn{\on{VE}(\phi)(X_1,\ldots,X_n,\ol{X}_{n+1},\ldots,\ol{X}_p)}\\&=\iota^*\sum_{s\in\mf{S}_p} \eps(s) 
\L(X_{s(1)}^{1,\sharp})\cdots \L(X_{s(n)}^{n,\sharp})\i(X_{s(n+1)}^{n+1,\sharp})\cdots \i(X_{s(p)}^{p,\sharp})\phi.
\end{split}\]
Here $\iota\colon M\to B_pG$ is the inclusion as constant $p$-arrows, and $\eps(s)$ is equal to $+1$ if the number of pairs 
$(i,j)$ with $1\le i< j\le n$ but $s(i)>s(j)$ is even, and equal to $-1$ if that number is odd.  
\end{theorem}
Observe that the formula does not involve the generating vector fields for the $i=0$ action. 
\begin{remarks}
\begin{enumerate}
\item 
This formula is similar to the expression obtained in Abad-Crainic \cite[Proposition 4.1]{aba:wei}. However, in contrast to 
the result in \cite{aba:wei}, no recursion procedure is needed.
\item 
The same formula holds true for local Lie groupoids, using the complex $\Om^\bullet(B_\bullet G)_M$ of germs.
\item 
Restricting, we obtain the following formula for the 
Van Est map $C^\infty(B_\bullet G)\to \sC^\bullet(A)$:
\[ \on{VE}(f)(X_1,\ldots,X_r)=\sum_{s\in\mf{S}_r} \on{sign}(s) \L(X_{s(1)}^{1,\sharp})\cdots \L(X_{s(r)}^{r,\sharp})\, f\Big|_M\]
This is the formula given by Weinstein and Xu \cite{we:sy}.
\item Mehta points out in \cite[Section 6]{meh:sup}
that the formula in Theorem~\ref{th:explicit} can be obtained from that of Weinstein and Xu \cite{we:sy} (c.f. \cite[Definition 6.2.1]{meh:sup}), via an appropriate modification to the signs due to the Koszul sign rule. 
\end{enumerate}
\end{remarks}

The proof will require some preparation. 
To simplify notation, denote by  $\ol{\otimes}:=\otimes_{\Om(M)}$ the (algebraic) tensor product of modules over commutative graded algebra $\Om(M)$. We will use the pullback  $\sz^*$ to regard $\Om(G)$ as an $\Om(M)$-module;
there is a natural multiplication map (not to be confused with cup product)
\begin{subequations}
\begin{equation}\label{eq:botimes1} \Om^{q_0}(G)\botimes \cdots \botimes \Om^{q_r}(G)\to \Om^{q_0+\ldots+q_r}(E_rG),\ \ \ 
\phi_0\botimes\cdots \botimes \phi_r\mapsto \pr_0^*\phi_0\cdots \pr_r^*\phi_r.\end{equation}
The Weil algebra $\sW^{\bullet,\bullet}(A)$ is also a module over $\Om(M)$;
the pullback $\pi_r^*$ defines an embedding as a subspace of $\sW^{\bullet,\bullet}(T_\F E_rG)$. We 
obtain an injective map, with dense image
\begin{equation}\label{eq:botimes}
\Omega^{q_0}(G)\bar{\otimes}\cdots \bar{\otimes} \Omega^{q_r}(G)\botimes \sW^{p,q}(A)
\to \sW^{p,q_0+\ldots+q_r+q}(T_\F E_rG)
\end{equation}
\end{subequations}
For $\phi_i\in \Om(G)$ and $\alpha\in \sW(A)$, we will identify $\phi_0\botimes\cdots\botimes \phi_r\botimes\alpha$ with its image under this map. 
On the image of this map, the homotopy operator $h$, the differential $\d'=(-1)^r\d_{CE}$, and the homomorphism $R_\bullet=\pi_\bullet^*\circ \iota_\bullet^*$ 
read as 
\[ \begin{split}
h\big(\phi_0\bar{\otimes}\cdots \bar{\otimes}\phi_r\botimes \alpha\big)
=&\sum_{i=0}^{r-1}(-1)^{i+1}
\phi_0\bar{\otimes}\cdots \bar{\otimes}\phi_i\bar{\otimes}\underbrace{1 \bar{\otimes}\cdots \bar{\otimes}1}_{r-i-1}\botimes
\iota^*(\phi_{i+1}\cdots \phi_r)\alpha,\\
\d'(\phi_0\botimes\cdots \botimes \phi_r\botimes \alpha)=&(-1)^{q_0+\ldots+q_r}\sum_{j=0}^r \sum_\nu \phi_0\botimes \cdots\botimes
\L(X_\nu^L)\phi_j\botimes \cdots\botimes \phi_r\botimes \beta^\nu\alpha\\
&+(-1)^{q_0+\ldots+q_r+r} \phi_0\botimes\cdots \botimes \phi_r\botimes \d_{CE} \alpha\\
R\big(\phi_0\bar{\otimes}\cdots \bar{\otimes}\phi_r\botimes \alpha\big)
=&\big(\underbrace{1\botimes\cdots\botimes 1}_{r+1})\botimes \iota^*(\phi_0\cdots \phi_r)\alpha
\end{split}\]
Here the second formula is to be understood locally, in terms of a local frame $X_1,\ldots,X_k$ of sections of $A$, with dual sections 
$\beta^1,\ldots,\beta^k$ of $A^*$.  The last two formulas imply that 
\begin{equation}\label{eq:nagut}
[\d',R]\big(\phi_0\bar{\otimes}\cdots \bar{\otimes}\phi_r\botimes \alpha\big)=(-1)^r
\big(\underbrace{1\botimes\cdots\botimes 1}_{r+1})\botimes \D(\phi_0\cdots \phi_r)\alpha
\end{equation}
The following formula involves the restriction $\D\colon \Om^q(G)\to \sW^{1,q}(A)$ of the map \eqref{eq:D}.
\begin{proposition}\label{prop:formula}
We have the following formula, for $\phi_i\in \Om^{q_i}(G)$ and $\alpha\in  \sW^{p,q}(A)$
\begin{equation}\label{eq:formula1}
 [\d',h]\big(\phi_0\bar{\otimes}\cdots \bar{\otimes}\phi_r\botimes \alpha\big)=(-1)^r\sum_{i=0}^{r-1}(-1)^{i+q_0+\ldots+q_i} 
\phi_0\bar{\otimes}\cdots \bar{\otimes}\phi_{i}\bar{\otimes}\underbrace{1\bar{\otimes}\cdots \bar{\otimes}1}_{r-i-1}\botimes  
(\D(\phi_{i+1}\cdots \phi_r)\,\alpha).\end{equation}
\end{proposition}
\begin{proof}
Using that $h$ is an $R$-derivation, one obtains the following property of $[\d',h]$ under cup product:
\begin{equation}\label{eq:yada1}
 [\d',h](x\cup y)=[\d',h]x\cup R y +x\cup [\d',h]y
-(-1)^{|x|} hx\cup [\d',R] y.
\end{equation}
for $x,y\in \sW^{\bullet,\bullet}(T_\F E_\bullet G)$. Here $|x|$ denotes the total degree of $x$.
In particular, take $x=\phi_0\bar\otimes 1$, as in \eqref{eq:botimes1}, with $\phi_0\in\Om^{q_0}(G)$. We have $|x|=q_0+1,\ \ hx=-\phi_0,\ \ [\d',h]x=0$, 
and 
\[ x\cup y=(-1)^{q_0 m} \phi_0\botimes y\] 
for $y\in \sW^{\bullet,\bullet}(T_\F E_mG)$. Hence 
the formula \eqref{eq:yada1} gives
\[ [\d',h](\phi_0\botimes y)=(-1)^{q_0}\phi_0\botimes [\d',h]y-(-1)^{q_0(m-1)}\phi_0\cup [\d',R]y.\]
If $y=\phi_1\botimes\cdots \botimes \phi_r\botimes\alpha\in\sW(T_\F E_{r-1}G)$, then we obtain, using \eqref{eq:nagut}, 
\[ [\d',R]y=(-1)^{r-1}\underbrace{1\botimes\cdots\botimes 1}_r\botimes \D(\phi_1\cdots \phi_r)\alpha.\]
Hence we find 
\[ [\d',h](\phi_0\botimes y)=(-1)^{q_0}\phi_0\botimes [\d',h]y+(-1)^r(-1)^{q_0}\phi_0\botimes 1\botimes \cdots \botimes 1\botimes 
\D(\phi_1\cdots \phi_r)\alpha,\]
which proves the Proposition.\end{proof}

\begin{proposition}
For $\phi_0,\ldots,\phi_r\in \Om(G)$ and $\alpha\in \sW^{p,q}(A)$, we have
\begin{equation}\label{eq:for}
\begin{split}
\lefteqn{\iota_0^* \circ (1+[\d',h])^{-1}\big(\phi_0\botimes\cdots \botimes \phi_r\botimes\alpha\big)}\\&=(-1)^{rq_0+(r-1)q_1+\ldots+q_{r-1}} 
(\iota_0^* \phi_0) (\D \phi_1)\cdots (\D \phi_r)\,\alpha \in \sW^{p,q+q_0+\ldots+q_r}(A).\end{split}\end{equation}
\end{proposition}
\begin{proof}
Using induction on $r$, we use Proposition \ref{prop:formula} to prove 
\begin{equation}\label{eq:tosh}
 [\d',h]^r \big(\phi_0\botimes\cdots \botimes \phi_r\botimes\alpha\big)=
(-1)^{r+rq_0+(r-1)q_1+\ldots+q_{r-1}} \phi_0\botimes ((\D \phi_1)\cdots (\D \phi_r)\, \alpha).
\end{equation}
For $r=1$ this is just a special case of Proposition \ref{prop:formula}. 
For  $r>1$, we apply the induction hypothesis for $r'=r-1$ to the formula for 
$[\d',h]\big(\phi_0\botimes\cdots \botimes \phi_r\botimes\alpha\big)$, as given in Proposition \ref{prop:formula}.
 Only the term with $i=r-1$ gives a nonzero contribution, and yields
\eqref{eq:tosh}. 
\end{proof}
\begin{remark}
The result \eqref{eq:for} may also be written
\[ (\iota_0^*\otimes \D\otimes \cdots \otimes \D \otimes \on{id})(
\phi_0\otimes\cdots \otimes \phi_r\otimes\alpha),\]
followed by the multiplication map 
$W(A)\otimes\cdots \otimes W(A)\to W(A)$. The signs appear naturally here, according to the super-sign rule: 
The first $\D$ moves past $\phi_0$, 
the second $\D$ moves past $\phi_0,\phi_1$, and so on. Hence we obtain $q_0+(q_0+q_1)+
\ldots +(q_0+\ldots+ q_{r-1})=rq_0+(r-1)q_1+\ldots+q_{r-1}$ sign changes. 
\end{remark}

\begin{proof}[Proof of Theorem \ref{th:explicit}]
Given $X_1,\ldots,X_r\in \Gamma(A)$ and any $n\le r$ we obtain, for all $\phi_0,\ldots,\phi_r\in \Om(G)$, 
\[ \begin{split} \lefteqn{  \big(\iota_0^*\circ (1+[\d',h])^{-1}(\phi_0\botimes\cdots \botimes \phi_r\botimes 1)\big)(X_1,\ldots,X_n,\ol{X}_{n+1},\ldots,\ol{X}_r)
}\\
&= (-1)^{rq_0+\ldots+q_{r-1}}\i_S(X_r)\cdots \i_S(X_{n+1})\i_K(X_n)\cdots \i_K(X_1) \big(\iota^*\phi_0 \,\D\phi_1\cdots \D\phi_r\big)\\
&=\iota_r^*\Big(\big(\L(-X_1^{1,R})\cdots \L(-X_n^{n,R})\i(-X_{n+1}^{n+1,R})\cdots \i(-X_r^{r,R})+\ldots \mathsf{s.p.}\big)
(\phi_0\botimes \cdots \botimes \phi_r)\Big); 
\end{split}
\]
here the lower dots signify a signed permutation of the $X_i$'s. Consequently, for $\phi\in \Om(B_rG)$ 
this gives 
\[ 
\begin{split}
\lefteqn{\big(\iota_0^*\circ (1+[\d',h])^{-1}\circ \kappa_r^*(\phi)\big)(X_1,\ldots,X_n,\ol{X}_{n+1},\ldots,\ol{X}_r)}\\
&=\iota_r^*\sum_{s\in \mf{S}_r} \eps(s)
\L(-X_{s(1)}^{1,R})\cdots \L(-X_{s(n)}^{n,R})\i(-X_{s(n+1)}^{n+1,R})\cdots \i(-X_{s(r)}^{r,R})\kappa_r^*\phi
\end{split}
\]
Here the sign $\eps(s)$ is the sign of the permutation putting $s(1),\ldots,s(n)\subset \{1,\ldots,r\}$ in order; 
in other words, it is $1$ if the number of pairs $1\le i<j\le n$ with $s(i)>s(j)$ is even, and is 
$-1$ if that number is odd. This implies the formula given in Theorem \ref{th:explicit}, because $-X^{i,R}$ is $\kappa_r$-related to 
$X^{i,\sharp}$. 
\end{proof}

\begin{example}
Let us examine these calculations for the case of a pair groupoid $G=\on{Pair}(M)=M\times M$. Here $\on{Lie}(G)=TM$, and for $X\in \Gamma(TM)=\mf{X}(M)$ we have 
\[ X^L=(0,X),\ \ X^R=(-X,0).\] The map $\D\colon C^\infty(M\times M)\to \sC^1(TM)=\Om^1(M)$
is given by 
\[ \D(u\otimes u')=-u'\d u,\ \ u,u'\in C^\infty(M).\]

We identify $B_pG=M^{p+1}$, where the $p+1$-tuple $(m_0,\ldots,m_p)$ corresponds to $(g_1,\ldots,g_p)$ with $g_i=(m_{i-1},m_i)$.
Similarly, $E_pG=M^{p+1}\times M$, where $(m_0,\ldots,m_p,m)$ corresponds to $(a_0,\ldots,a_p)$ with 
$a_i=(m_i,m)$. Given $u_0\otimes \cdots \otimes u_p\in C^\infty(M^{p+1})$ with $u_i\in C^\infty(M)$, the pullback to 
$E_pG$ is $f_0\botimes \cdots \botimes f_p$ with $f_i=u_i\otimes 1$, with $\D(f_i)=-\d u_i$. 
Thus
\[ \iota_0^*\circ (1+[\d,h])^{-1}(f_0\botimes \cdots \botimes f_p)=(-1)^pu_0 \d u_1 \cdots \d u_p.\]
Hence the Van Est map becomes (up to a sign)
the standard map from the Alexander-Spanier complex to the de Rham complex:
\[ \on{VE}\colon C^\infty(M^{p+1})\to \Om^p(M),\ \ \ \ u_0\otimes \cdots \otimes u_p\mapsto (-1)^p u_0 \d u_1\cdots \d u_p\]
\end{example}

\begin{appendix}
\section{Simplicial manifolds}\label{app:simplicial}\label{app:A}
In this section we give a quick review of simplicial techniques used in this paper. Standard references include Bott-Mostow-Perchik \cite{mo:no}, Goerss-Jardine \cite{goe:sim}. 
%
\subsection{Basic definitions}
Let $\on{Ord}$ denote the category of ordered sets. The objects
in $\on{Ord}$ are $[0],[1],[2],\ldots$, where 
$[n]=\{0,\ldots,n\}$, and the morphisms in $\on{Ord}$ are the 
maps $f\colon [m]\to [n]$ such that $i\le j\Rightarrow f(i)\le f(j)$.
Any such morphism may be written as a composition of 
\emph{face maps} $\partial^j$ 
\emph{degeneracy maps} $\eps^j$, 
\[ \partial^j\colon [n]\to [n+1],\ \ j=0,\ldots,n+1,\ \ \ 
\eps^j\colon [n+1]\to [n],\ j=0,\ldots,n
\]
given by 
\[ \partial^j(i)=\begin{cases} i & i< j\\
i+1 & i\ge j\end{cases},\ \ \eps^j(i)=
\begin{cases}
i & i\le j\\
i-1 & i> j.
\end{cases}
\]
A \emph{simplicial manifold} is a contravariant functor from the category $\on{Ord}$ to the category of manifolds. We denote by $X_n$ the image of $[n]=\{0,\ldots,n\}$, and by 
$X(f)\colon X_n\to X_m$ the map corresponding to a morphism $f\colon [m]\to [n]$. We will write $\partial_i:=X(\partial^i)$, and $\eps_i:=X(\eps^i)$. 
%
Associated to any topological category $C$ is a simplicial space $B_\bullet C$, called its simplicial classifying space 
(or \emph{nerve}) \cite{se:cl}. Here $B_0C$ is the set of objects of the category, $B_1C$ the set of arrows (morphisms in $C$), 
$B_2C$ the set of commutative triangles, and so on. 
\begin{example}
If $G\rra M$ is a Lie groupoid  (regarded as a category), the space $B_pG$ is the manifold of 
$p$-arrows, as in Section  \ref{subsec:liegr}. 
\end{example}
%
\begin{example}\label{ex:simplices}
For any fixed $p$, the set $[p]=\{0,\ldots,p\}$ may be regarded as the objects of a category, with a unique arrow $i_0\longleftarrow i_1$ for any $0\le i_0\le i_1\le p$. 
The corresponding space $B_n[p]$ is the set of $n$-arrows of this type, 
\[ i_0\longleftarrow i_1\longleftarrow\cdots \longleftarrow i_n\]
where $0\le i_0\le \cdots \le i_n\le p$. Equivalently, $B_n[p]$ is the set of nondecreasing maps $[n]\to [p]$.  
Any morphism  $[m]\to [n]$ in the category $\on{Ord}$ determines a simplicial map $B_n[p]\to B_m[p]$ 
for the category $[p]$,  by composition. 
We will denote this (discrete) simplicial manifold by $\Delta_\bullet[p]:=B_\bullet[p]$, since its geometric realization is the standard $p$-simplex. 
Any nondecreasing map $[p]\to [p']$ defines a morphism of simplicial manifolds $\Delta_\bullet[p]\to \Delta_\bullet[p']$, with geometric realization the corresponding map of standard simplices. 
\end{example}

%
%
%

\subsection{Simplicial homotopies}\label{app:A-simhom}
The two morphisms $\partial^0,\partial^1\colon [0]\to [1]$ give rise to  simplicial maps 
\[ \partial^0_\bullet,\ \partial^1_\bullet\colon \Delta_\bullet[0]\to \Delta_\bullet[1],\] 
corresponding to the inclusions of the end points.
A \emph{simplicial homotopy} between  two morphisms of simplicial manifolds $f_\bullet^0,f_\bullet^1\colon X_\bullet\to Y_\bullet$ is a morphism
\[ H_\bullet\colon \Delta_\bullet[1]\times X_\bullet \to Y_\bullet\]
such that 
\[ H_\bullet\circ (\partial^0_\bullet\times \on{id}_{X_\bullet})=f^0_\bullet,\ \ \ 
H_\bullet\circ (\partial^1_\bullet\times \on{id}_{X_\bullet})=f^1_\bullet.\]
Homotopy is an equivalence relation provided $X_\bullet$ satisfies the \emph{Kan condition} \cite{goe:sim}; in particular this is the 
case for the simplicial classifying space of a groupoid. 
To spell out the homotopy condition in more detail, note 
that $\Delta_p[1]=\{\alpha_{-1},\alpha_0,\ldots,\alpha_p\}$ with
\[\alpha_j\colon [p]\to [1],\ \ \ \alpha_j(i)=\begin{cases} 0 & i\le j\\ 1 & i>j\end{cases},\] 
hence $H_p$ is determined by the maps $H_{p,j}=H_p(\alpha_j,\cdot)$ for $-1\le j\le p$. The condition 
that $H_\bullet$ be a simplicial map becomes 
\[ \partial_i\circ H_{p,j}=\begin{cases} H_{p-1,j-1} \circ \partial_i & i\le j \\H_{p-1,j} \circ \partial_i & i>j \end{cases},\ \ \ 
\eps_i\circ H_{p,j}=\begin{cases} H_{p+1,j+1}\circ \eps_i&i\le j \\ H_{p+1,j}\circ \eps_i & i> j\end{cases},\]
and the boundary conditions are
\[ H_{p,-1}=f_p^0,\ \ H_{p,p}=f_p^1.\]
The map $(\partial^0)_p$ takes the unique element of $\Delta_p[0]$ to $\alpha_{-1}$, while $(\partial^1)_p$ takes it to 
$\alpha_p$.

Associated to any simplicial space $X$ is its \emph{Moore complex} 
$(\Z X_\bullet,\delta)$, where $\Z X_p$ are $\Z$-linear combinations of elements in $X_p$, and 
\[ \delta_p=\sum_{j=0}^{p}(-1)^j\partial_j\colon \Z X_p\to \Z X_{p-1}.\]

Any simplicial homotopy gives rise to a homotopy operator for the Moore complexes, by the formula
%
%
%
\begin{equation}\label{eq:homotopie}
h_p\colon  \Z X_p\to \Z Y_{p+1},\ \ h_p=\sum_{j=0}^p(-1)^{j+1} h_{p,j}
\end{equation}
with $h_{p,j}(x)=H_{p+1,j}(\eps_j(x))$. That is, $h_\bullet$ satisfies 
$ h_{p-1}\partial_p +\partial_{p+1} h_p=f^0_p-f^1_p$. 
See Goerss-Jardine \cite[Lemma 2.15]{goe:sim}.

For the following result, recall that for any foliation $\ca{F}$ of a manifold $M$, the groupoid $\on{Pair}_\F(M)\rra M$ consists of pairs $(m_0,m_1)$ of elements in the same leaf, and $B_p\on{Pair}_\F(M)$ consists of $p+1$-tuples $(m_0,\ldots,m_p)$ of elements $m_i\in M$, all in the same leaf. 
Any smooth map $f\colon M\to M$ preserving leaves extends to a simplicial map 
\begin{equation}\label{eq:fbullet} f_\bullet\colon B_\bullet\on{Pair}_\F(M)\to B_\bullet\on{Pair}_\F(M)\end{equation}
where $f_p(m_0,\ldots,m_0)=(f(m_0),\ldots,f(m_p))$. 
The following result may be regarded as a special case of \cite[Proposition 2.1]{se:cl}. The proof is a straightforward verification.
\begin{proposition}\label{prop:fol}
Let $\ca{F}$ be a foliation of a manifold $M$, and $f\colon M\to M$ a smooth map preserving leaves. 
Then 
\begin{equation}\label{eq:simplicialhomotopy}
H_{p,j}(m_0,\ldots,m_p)=(m_0,\ldots,m_j,f(m_{j+1}),\ldots,f(m_p)),
\end{equation}
defines a a simplicial homotopy $H_\bullet$ between \eqref{eq:fbullet} and the identity map. The corresponding 
homotopy operator is given by $h_p=\sum_{j=0}^p(-1)^{j+1} h_{p,j}\colon \Z B_p\on{Pair}_\F(M) \to \Z B_{p+1}\on{Pair}_\F(M)$ where (cf.~ \eqref{eq:homotopie}) 
\[  h_{p,j}(m_0,\ldots,m_p)=(m_0,\ldots,m_j,f(m_j),f(m_{j+1}),\ldots,f(m_p)).\]
Thus, $h_{p-1}\circ \partial_p+\partial_{p+1}\circ h_p=\on{id}-f_p$. If $f$ is a retraction (i.e., $f\circ f=f$), then the homotopy operator has the additional property $h_{p+1}\circ h_p=0$.
\end{proposition}
%

We will use the following special case: Suppose $\pi\colon Q\to M$ is a surjective submersion  admitting a section $\iota\colon M\to Q$. 
The submersion defines a foliation of $Q$, where $B_p\on{Pair}_\F Q$ is the $p+1$-fold fiber product $Q^{(p+1)}=Q\times _M\cdots \times_M Q$. Take $f=\iota\circ \pi\colon Q\to Q$. The proposition shows that the two maps 
\[ \pi_\bullet \colon Q^{(\bullet+1)}\to M,\ \ \iota_\bullet \colon M\to Q^{(\bullet+1)}\]
are simplicial homotopy inverses, with an explicit homotopy operator
\[ h_p(q_0,\ldots,q_p)=\sum_{i=0}^p (-1)^i (q_0,\ldots,q_i,m,\ldots,m)\]
where $m=\pi(q_0)=\ldots=\pi(q_p)$.
%
%

\section{Homological perturbation theory}\label{app:perturb}
In this paper we used the following two results, Lemmas \ref{lem:A} and \ref{lem:B}, which are special cases of results from homological perturbation theory. 

Let $(C^{\bullet,\bullet},\d,\delta)$ be a double complex, with differentials $\delta$ of bidegree $(0,1)$ and $\d$ of bidegree $(1,0)$
so that $[\d,\delta]=\d\delta+\delta\d=0$. We assume that $C^{r,s}$ is non-zero only in degrees $r,s\ge 0$.
The corresponding total complex is given by $\on{Tot}^\bullet C=\bigoplus_{r+s=\bullet} C^{r,s}$
with the total differential $\d+\delta$. Suppose that 
\[ i\colon D^{\bullet,\bullet}\hra C^{\bullet,\bullet}\] 
is a subcomplex for both differentials $\d$ and $\delta$, and that there exists an operator $h$ of bidegree $(-1,0)$ 
such that\footnote{In what follows, the brackets $[\cdot,\cdot]$ indicate graded commutators for the total degree.}
\[ [h,\delta]=h\delta+\delta h=1-i\circ p,\]
with $p\colon  C^{\bullet,\bullet}\to D^{\bullet,\bullet}$ a left inverse to $i$.
This equation shows that $i$ is a homotopy equivalence with respect to $\delta$, with homotopy 
inverse $p$. Indeed, $p\circ i=\on{id}$, while the projection operator$\Pi=i\circ p$ is $\delta$-homotopic to the identity.
Note however that $p$ need not intertwine the differential $\d$. 

By the following result, one can modify $p$ and $i$ to obtain a homotopy equivalence 
for the total differential $\d+\delta$.
%
It is a version of the \emph{Basic Perturbation Lemma} \cite{bro:twi,gug:cha,gug:per,gug:per2,hue:sma}. 
See Crainic \cite{cra:dif} and Johnson-Freyd \cite{joh:hom} for some recent applications.
\begin{lemma}[Brown \cite{bro:twi}, Gugenheim \cite{gug:cha}]\label{lem:A}
Put $p'=p(1+\d h)^{-1},\ \ i'=(1+h\d)^{-1}i,\ \ h'=h(1+\d h)^{-1}$. Then:
\begin{enumerate}
\item 
The map $\Pi'=i'p' $ 
is a cochain map relative to the total differential $\d+\delta$. In fact, it is homotopic to the identity 
with the homotopy operator $h'$:
\[
[h',\d+\delta]=1-i'p'.\]
\item 
If $h$ preserves the subcomplex $D$, and commutes with $\d$ on $D$, 
then $\Pi'$ is again a projection onto $D$. Furthermore, in this case 
$p'$ is a cochain map with respect to the total differential, and  
is a homotopy equivalence, with homotopy inverse $i'$.
\end{enumerate}
\end{lemma}
\begin{proof}
\begin{enumerate}
\item
We have $(1+hd)h'=h=h'(1+\d h)$, hence 
\[ (1+h\d)[h',\d+\delta](1+\d h)=h(\d+\delta)(1+\d h)+(1+h\d)(\d+\delta)h
=[h,\d+\delta].\]
where we used $\d\delta+\delta\d=0$. On the other hand, 
$[h,\delta]=1-ip$ implies
\[ (1+h\d) (1-i'p')(1+\d h)=[h,\d+\delta].\]
Comparing the two formulas, we see $[h',\d+\delta]=1-i'p'$. 
\item 
We have 
\[ p'i'=p(1+h \d)^{-1}(1+\d h)^{-1}i=p(1+[\d,h])^{-1}i.\]
Hence, if $[\d,h]$ vanishes on $D$, then $p'i'=pi=1$ so that $\Pi'=i'p'$ is again a projection. 
If $h$ preserves $D$, so that $(1+h\d)$ restricts to an invertible transformation of $D$, we 
see that $\Pi'$ has 
the same range as $\Pi$. Since $\Pi'$ is a cochain map with respect to $\d+\delta$, 
the same is true of $p'$. 
\end{enumerate}
\end{proof}
%
\begin{remark}
The second part of this Lemma applies in particular if $h$ vanishes 
on $D$. Note also that if $h^2=0$, then $D$ is preserved by $h$, since  $[h,\Pi]=[h,1-[h,\delta]]=0$. 
\end{remark}

%
Let us now make the additional assumption that bidifferential space $C^{\bullet,\bullet}$ has a compatible algebra structure
$\phi\otimes \psi\mapsto \phi\cup\psi$, with $D^{\bullet,\bullet}$ a subalgebra. 
Thus, in particular $\d$ and 
$\delta $ are derivations of this algebra structures. We also assume that the projection $p$ is an algebra morphism
and that $(C_{\on{tot}}^\bullet,\d+\delta )$ is a differential algebra. 
\begin{lemma}[Gugenheim-Lambe-Stasheff \cite{gug:per2}]\label{lem:B} 
Suppose the homotopy operator $h$ is a \emph{$\Pi$-derivation}, that is, 
\[ h(\phi\cup \psi)=h\phi\cup \Pi\psi+(-1)^{|\phi|}\phi\cup h\psi.\]
Assume furthermore that $h$ satisfies the `side conditions'
\[ h\circ h=0,\ \ p\circ h=0\]
Then the map $\Pi'=\Pi (1+\d h)^{-1}\colon C^\bullet_{\on{tot}}\to D^\bullet_{\on{tot}}\subseteq  C^\bullet_{\on{tot}}$ is a morphism of differential algebras.
\end{lemma}
\begin{proof}
Observe that $hh=0$ implies that $h$ commutes with $\Pi=1-[h,\delta ]$. Hence, $ph=0 \Rightarrow \Pi h=0\Rightarrow 
h\Pi=0\Rightarrow hi=0$. That is, $h$ vanishes on $D$. It follows that $i'=i$, hence  $\Pi'=\Pi(1+\d h)^{-1}=
\Pi(1+[\d,h])^{-1}$. With $H=[\d,h]$, we obtain 
\[ \Pi'=\Pi(1+H)^{-1}=\sum_{k=0}^\infty (-1)^k \Pi H^k.\]
The $\Pi$-derivation property of $h$ implies the following property of $H$:
\[ H(\phi\cup\psi)=H\phi\cup \Pi\psi+\phi\cup H\psi+(-1)^{|\phi|+1} h\phi\cup [\d,\Pi]\psi.\]
Iteration of this formula, using $H\Pi=0$ and $[H,h]=0$, gives 
\[ H^k(\phi\cup \psi)=\sum_{j=0}^k H^{k-j} \phi \cup \Pi^{k-j} H^{j} \psi +\sum_\nu h\phi^{(k)}_\nu \cup \psi^{(k)}_\nu\]
with certain elements $\phi^{(k)}_\nu,\psi^{(k)}_\nu$. Now apply the projection $\Pi$. Since $\Pi$ is an algebra morphism, and 
$\Pi h=0$ and $H\Pi=0$, we obtain 
\[ \Pi H^k(\phi\cup \psi)=\sum_{j=0}^k \Pi H^{k-j} \phi \cup \Pi H^{j} \psi ,\]
which gives $\Pi'(\phi\cup\psi)=\Pi'\phi\cup \Pi'\psi$ as desired. 
\end{proof}
\begin{remark}\label{rem:gug}
The same proof also gives the following more general statement, applicable to bilinear maps of bidifferential spaces.
We will again write this bilinear map as a `cup' product, although it might be for example a module action, a Lie bracket, etc.
Thus suppose 
\[ \cup \colon C_1\otimes C_2\to C_3\] 
is a morphism of bidifferential spaces. Suppose that $\cup$ restricts to a bilinear map on subcomplexes
$i_\nu\colon D_\nu\hra C_\nu$, that $p_\nu\colon D_\nu\to C_\mu$ are compatible with $\cup$ in the sense that 
$p_3(\phi\cup\psi)=p_1(\phi)\cup p_2(\psi)$, and that we are given homotopy operators $h_\nu$ 
for the $\delta$-differentials, i.e.,  
\[ [h_\nu,\delta]=1-i_\nu p_\nu.\]
If $h_\nu$ have the `derivation property'
\[ h_3(\phi\cup \psi)=h_1\phi\cup \Pi_2\psi+(-1)^{|\phi|}\phi\cup h_2\psi\]
for $\phi\in C_1,\ \psi\in C_2$, and if the side conditions $h_\nu^2=0$ and $p_\nu h_\nu=0$
are satisfied, then $\Pi_\nu'=\Pi_\nu(1+\d h_\nu)^{-1}$ are cochain maps for the total differentials, with 
\[ \Pi_3'(\phi\cup\psi)=\Pi_1'(\phi)\cup \Pi_2'(\psi).\]
\end{remark}

\end{appendix}


\def\cprime{$'$} \def\polhk#1{\setbox0=\hbox{#1}{\ooalign{\hidewidth
  \lower1.5ex\hbox{`}\hidewidth\crcr\unhbox0}}} \def\cprime{$'$}
  \def\cprime{$'$} \def\cprime{$'$}
  \def\polhk#1{\setbox0=\hbox{#1}{\ooalign{\hidewidth
  \lower1.5ex\hbox{`}\hidewidth\crcr\unhbox0}}} \def\cprime{$'$}
  \def\cprime{$'$} \def\cprime{$'$} \def\cprime{$'$} \def\cprime{$'$}
\providecommand{\bysame}{\leavevmode\hbox to3em{\hrulefill}\thinspace}
\providecommand{\MR}{\relax\ifhmode\unskip\space\fi MR }
\providecommand{\MRhref}[2]{%
  \href{http://www.ams.org/mathscinet-getitem?mr=#1}{#2}
}
\providecommand{\href}[2]{#2}


\begin{thebibliography}{40}

\bibitem{aba:th}
C.~A.~Abad, \emph{Representations up to homotopy and cohomology of classifying
  spaces}, Ph.D. thesis, Utrecht University, 2008.

\bibitem{aba:wei}
C.~A.~Abad and M.~Crainic, \emph{The {W}eil algebra and the {V}an {E}st
  isomorphism}, Ann. Inst. Fourier (Grenoble) \textbf{61} (2011), no.~3,
  927--970.

\bibitem{bo:rh}
R.~Bott, H.~Shulman, and J.~Stasheff, \emph{On the de {R}ham theory of certain
  classifying spaces}, Advances in Math. \textbf{20} (1976), no.~1, 43--56.

\bibitem{bro:twi}
R.~Brown, \emph{The twisted {E}ilenberg-{Z}ilber theorem}, Simposio di
  {T}opologia ({M}essina, 1964), Edizioni Oderisi, Gubbio, 1965, pp.~33--37.

\bibitem{bur:mul}
H.~Bursztyn and A.~Cabrera, \emph{Multiplicative forms at the
  infinitesimal level}, Math. Ann. \textbf{353} (2012), no.~3, 663--705.

\bibitem{bur:lin}
H.~Bursztyn, A.~Cabrera and C.~Ortiz, \emph{Linear and
  multiplicative 2-forms}, Lett. Math. Phys. \textbf{90} (2009), no.~1-3,
  59--83.

\bibitem{bur:intdir}
H.~Bursztyn, M.~Crainic, A.~Weinstein and C.~Zhu,
  \emph{{Integration of twisted Dirac brackets}}, Duke Mathematical Journal
  \textbf{123} (2004), no.~3, 549--607.

\bibitem{con:ind}
A.~Connes and H.~Moscovici, \emph{{Differentiable cyclic cohomology and
  Hopf algebraic structures in transverse geometry}}, Essays on geometry and
  related topics, Vol. 1, 2, Enseignement Math., Geneva, 2001, pp.~217--255.
  
\bibitem{cra:dif}
M.~Crainic, \emph{Differentiable and algebroid cohomology, van {E}st
  isomorphisms, and characteristic classes}, Comment. Math. Helv. \textbf{78}
  (2003), no.~4, 681--721.

\bibitem{cra:intlie}
M.~Crainic and R.~Fernandes, \emph{{Integrability of Lie brackets}},
  Annals of Mathematics. Second Series \textbf{157} (2003), no.~2, 575--620.

\bibitem{cra:intpoi}
\bysame, \emph{{Integrability of Poisson brackets}}, Journal of Differential
  Geometry \textbf{66} (2004), no.~1, 71--137.

\bibitem{cra:lect}
\bysame, \emph{Lectures on integrability of {L}ie
  brackets}, 2006, Preprint, arXiv:math/0611259.
  
\bibitem{cra:lin}
\bysame, \emph{{A geometric approach to Conn's linearization theorem}}, Annals of Mathematics. Second Series \textbf{173} (2011), no.~2, 1121--1139 (English).

\bibitem{cra:exi}
M.~Crainic and I.~Marcut, \emph{On the existence of symplectic realizations},
  J. Symplectic Geom. \textbf{9} (2011), no.~4, 435--444.

\bibitem{cra:spenc}
M.~Crainic, M.~A.~Salazar, and I.~Struchiner,
  \emph{{Multiplicative Forms and Spencer Operators}}, 2012, Preprint, {arXiv:math/1210.2277}.

\bibitem{cran:jacob}
M.~Crainic and C.~Zhu, \emph{{Integrability of Jacobi and Poisson
  structures}}, Universit\'e de Grenoble. Annales de l'Institut Fourier
  \textbf{57} (2007), no.~4, 1181--1216.

\bibitem{vanEst:groupCoh}
W.~T.~van Est, \emph{{Group cohomology and Lie algebra cohomology in Lie groups.
  I, II}}, Nederl. Akad. Wetensch. Proc. Ser. A. 56 = Indagationes Math.
  \textbf{15} (1953), 484--492, 493--504.

\bibitem{vanEst:algCoh}
\bysame, \emph{{On the algebraic cohomology concepts in Lie groups. I, II}},
  Nederl. Akad. Wetensch. Proc. Ser. A. 58 = Indag. Math. \textbf{17} (1955),
  225--233, 286--294.

\bibitem{vanEst:cartLeray}
\bysame, \emph{{Une application d'une m\'ethode de Cartan-Leray}}, Nederl.
  Akad. Wetensch. Proc. Ser. A. 58 = Indag. Math. \textbf{17} (1955), 542--544.

\bibitem{goe:sim}
P.~G.~Goerss and J.~F.~Jardine, \emph{Simplicial homotopy theory}, Modern
  Birkh\"auser Classics, Birkh\"auser Verlag, Basel, 2009, Reprint of the 1999
  edition.

\bibitem{gra:vb}
A.~Gracia-Saz and R.~Mehta, \emph{{VB}-groupoids and representation theory of
  {L}ie groupoids}, arXiv:1007.3658.

\bibitem{gra:vba}
\bysame, \emph{Lie algebroid structures on double vector bundles and
  representation theory of lie algebroids}, Advances in Mathematics
  \textbf{223} (2010), 1236--1275.

\bibitem{gua:sym}
M.~Gualtieri and S.~Li, \emph{{Symplectic groupoids of log symplectic
  manifolds}}, arXiv:1206.3674.

\bibitem{gug:cha}
V.~K. A.~M. Gugenheim, \emph{On the chain-complex of a fibration}, Illinois J.
  Math. \textbf{16} (1972), 398--414.

\bibitem{gug:per}
V.~K.~A.~M.~Gugenheim and L.~A.~Lambe, \emph{Perturbation theory in
  differential homological algebra. {I}}, Illinois J. Math. \textbf{33} (1989),
  no.~4, 566--582. \MR{1007895 (91e:55023)}

\bibitem{gug:per2}
V.~K.~A.~M.~Gugenheim, L.~A.~Lambe, and J.~D. Stasheff, \emph{Perturbation
  theory in differential homological algebra. {II}}, Illinois J. Math.
  \textbf{35} (1991), no.~3, 357--373.

\bibitem{gu:sy}
V.~Guillemin and S.~Sternberg, \emph{Symplectic techniques in physics},
  Cambridge Univ.~Press, Cambridge, 1990.

\bibitem{hig:alg}
P.~J.~Higgins and K.~Mackenzie, \emph{Algebraic constructions in the category of
  {L}ie algebroids}, J.~Algebra \textbf{129} (1990), 194--230.

\bibitem{hue:sma}
J.~Huebschmann and T.~Kadeishvili, \emph{Small models for chain algebras},
  Math. Z. \textbf{207} (1991), no.~2, 245--280.

\bibitem{joh:hom}
T.~Johnson-Freyd, \emph{Homological perturbation theory for nonperturbative
  integrals}, 2012, arXiv: 1206.5319.


\bibitem{koc:dif}
D.~Kochan and P.~Severa, \emph{Differential gorms, differential worms}, Preprint,
  {arXiv:math/0307303}.

\bibitem{tu:ktheory}
C.~Laurent-Gengoux, J.-L.~Tu and P.~Xu, \emph{{Twisted K-theory of
  differentiable stacks}}, Annales Scientifiques de l{\textquoteright}{\'E}cole
  Normale Sup{\'e}rieure \textbf{37} (2004), no.~6, 841--910 (English).


\bibitem{lib:cou}
D.~Li-Bland and E.~Meinrenken, \emph{{C}ourant algebroids and {P}oisson
  geometry}, International Mathematics Research Notices \textbf{11} (2009),
  2106--2145.

\bibitem{mac:gen}
K.~Mackenzie, \emph{General theory of {L}ie groupoids and {L}ie algebroids},
  London Mathematical Society Lecture Note Series, vol. 213, Cambridge
  University Press, Cambridge, 2005.

\bibitem{meh:sup}
R.~Mehta, \emph{Supergroupoids, double structures, and equivariant cohomology},
  2006, Thesis (Ph.D.)--University of California, Berkeley.

\bibitem{me:clifbook}
E.~Meinrenken, \emph{{L}ie groups and {C}lifford algebras}, {Ergebnisse der
  Mathematik und ihrer Grenzgebiete}, vol.~58, Springer, Heidelberg, 2013.

\bibitem{mel:ati}
R.~B.~Melrose, \emph{The Atiyah-Patodi-Singer index theorem}, Research Notes in
  Mathematics, vol.~4, A K Peters, Ltd., Wellesley, 1993.

\bibitem{moe:fol}
I.~Moerdijk and J.~Mr{\v{c}}un, \emph{Introduction to foliations and {L}ie
  groupoids}, Cambridge Studies in Advanced Mathematics, vol.~91, Cambridge
  University Press, Cambridge, 2003.

\bibitem{mo:no}
M.~Mostow and J.~Perchik, \emph{Notes on {G}elfand-{F}uks cohomology and
  characteristic classes (lectures delivered by {R}.~{B}ott)}, Proceedings of
  the eleventh annual holiday symposium, New Mexico State University, 1973,
  pp.~1--126.

\bibitem{pfla:longind}
M.J.~Pflaum, H.~Posthuma and X.~Tang, \emph{{The localized longitudinal index
  theorem for Lie groupoids and the van Est map}},  2011, Preprint, {arXiv:math/1112.4857}.
  
\bibitem{pfla:geomind}
\bysame, \emph{{The index of geometric operators on Lie groupoids}}, 
  2013, Preprint, {arXiv:math/1308.0236}.

\bibitem{pfla:tra}
\bysame, \emph{{The transverse index theorem for proper cocompact actions of
  Lie groupoids}}, 2013, Preprint, {arXiv:math/1301.0479}.
\bibitem{pra:rem}
J.~Pradines, \emph{{Remarque sur le groupo\"\i de cotangent de
  Weinstein-Dazord}}, C. R. Acad. Sci. Paris S\'er. I Math. \textbf{306}
  (1988), no.~13, 557--560. 


\bibitem{sal:thesis}
M.~A.~Salazar, \emph{{Pfaffian groupoids}}, Ph.D. thesis,
 Universiteit Utrecht, June 2013.
 
\bibitem{se:cl}
G.~Segal, \emph{Classifying spaces and spectral sequences}, Inst. Hautes
  \'Etudes Sci. Publ. Math. (1968), no.~34, 105--112.

\bibitem{shu:th}
H.~Shulman, \emph{On characteristic classes}, 1972, Ph.D. thesis, Berkeley.

\bibitem{spa:al}
E.~H.~Spanier, \emph{Algebraic topology}, Springer-Verlag, New York, 1981.


\bibitem{vin:ite}
A.~M.~Vinogradov and L.~Vitagliano, \emph{Iterated differential forms:
  tensors}, Dokl. Akad. Nauk \textbf{407} (2006), no.~1, 16--18.
  
\bibitem{we:sy}
A.~Weinstein, \emph{The symplectic structure on moduli space}, The Floer
  memorial volume, Progr. Math., vol. 133, Birkh\"auser, Basel, 1995,
  pp.~627--635.

\bibitem{wei:ext}
A.~Weinstein and P.~Xu, \emph{Extensions of symplectic groupoids and
  quantization}, J. Reine Angew. Math. \textbf{417} (1991), 159--189.
  
\bibitem{wein:lin}
A.~Weinstein, \emph{{Linearization of regular proper groupoids}}, Journal de l'Institut de Math\'ematiques de Jussieu \textbf{1} (2002), no.~3, 493--511.
  
  

\end{thebibliography}

\end{document}